\documentclass[12pt]{amsart}
\usepackage[centertags]{amsmath}
\usepackage{amsfonts,amsthm}
\usepackage{amscd}
\usepackage{latexsym}
\usepackage{amssymb}
\usepackage{mathpazo}
\usepackage{eucal}

\begin{document}
\theoremstyle{plain} 
\newtheorem{theorem}{\em Theorem\,}[section]
\newtheorem{lemma}[theorem]{\em Lemma\,}
\newtheorem{corollary}[theorem]{\em Corollary\,}
\newtheorem{proposition}[theorem]{\bf Proposition}
\newtheorem{claim}[theorem]{\bf Claim}
\theoremstyle{definition} 
\newtheorem{definition}[theorem]{\em Definition}
\newtheorem{remark}[theorem]{\em Remark}
\newtheorem{example}[theorem]{\em Example}
\newtheorem{notation}[theorem]{\em Notation}

\font\medit=ptmri at 18pt
\def\vp{\tau}
\def\navp{\nabla\hskip-.5pt\vp}
\def\bbI{\mathbf{I}}
\def\bbR{\mathrm{I\!R}}
\def\bbRP{\bbR\mathrm{P}}
\def\bbC{{\mathchoice {\setbox0=\hbox{$\displaystyle\mathrm{C}$}\hbox{\hbox 
to0pt{\kern0.4\wd0\vrule height0.9\ht0\hss}\box0}} 
{\setbox0=\hbox{$\textstyle\mathrm{C}$}\hbox{\hbox 
to0pt{\kern0.4\wd0\vrule height0.9\ht0\hss}\box0}} 
{\setbox0=\hbox{$\scriptstyle\mathrm{C}$}\hbox{\hbox 
to0pt{\kern0.4\wd0\vrule height0.9\ht0\hss}\box0}} 
{\setbox0=\hbox{$\scriptscriptstyle\mathrm{C}$}\hbox{\hbox 
to0pt{\kern0.4\wd0\vrule height0.9\ht0\hss}\box0}}}} 
\def\bbCP{\bbC\mathrm{P}}
\def\dimr{\dim_{\hskip-.2pt\bbR\hskip-1.7pt}^{\phantom i}}
\def\dimc{\dim_{\hskip.4pt\bbC\hskip-1.2pt}}
\def\Lie{\pounds}
\def\hs{\hskip.7pt}
\def\hh{\hskip.4pt}
\def\hn{\hskip-.4pt}
\def\nh{\hskip-.7pt}
\def\nnh{\hskip-1.5pt}
\def\ns{\hskip-1.2pt}
\def\tm{{T\hskip-.3ptM}}
\def\yj{\gamma}
\def\sd{\varSigma}
\def\nv{E}
\def\ds{\lambda}
\def\vg{\varGamma}
\def\si{\phi}
\def\ta{\psi}
\def\ve{\varepsilon}
\def\lar{a}
\def\prj{\pi}
\def\pro{P}

\voffset=-25pt\hoffset=13pt  

\title[Kil\-ling potentials with geodesic gradients]{$\text{\medit 
Kil\-ling Potentials with Geodesic Gradients}$\\
\vskip2pt
$\text{\medit on K\"ah\-ler Surfaces}$}
\author{{\sc Andrzej Derdzinski}}
\address{\hbox{\rm Department\ of\ Mathematics}\linebreak
\hbox{\hskip11.6pt\rm The\ Ohio\ State\ University}\linebreak
\hbox{\hskip11.6pt\rm Columbus,\ OH\ 43210,\ USA}\linebreak
\hbox{{\hskip11.6pt\sc E-mail:}\ {\sf andrzej@math.ohio-state.edu}}}
\begin{abstract}
We classify compact K\"ah\-ler surfaces with nonconstant Killing potentials 
such that all integral curves of their gradients are reparametrized geodesics.
\end{abstract}



\maketitle

\section{Introduction}\label{in}
Let $\,\vp\,$ be a Kil\-ling potential on a K\"ah\-ler manifold $\,(M,g)$, by 
which one means a $\,C^\infty$ function $\,\vp:M\to\bbR\,$ such that 
$\,J(\navp)\,$ is a Kil\-ling field on $\,(M,g)$. We say that $\,\vp\,$ has a 
{\it geodesic gradient\/} if all nontrivial integral curves of $\,\navp\,$ are 
re\-pa\-ram\-e\-trized geodesics, or---equivalently (Section~\ref{se})---if 
$\,dQ\wedge\hs d\vp=0$, where $\,Q=g(\navp,\navp)$.

There are many known examples of nonconstant Kil\-ling potentials with 
geodesic gradients on compact K\"ah\-ler manifolds. They include the 
sol\-i\-ton functions of the K\"ah\-\hbox{ler\hs-}\hskip0ptRic\-ci sol\-i\-tons discovered by 
Koiso \cite{koiso} and, independently, Cao \cite{cao}; special 
K\"ah\-\hbox{ler\hs-}\hskip0ptRic\-ci potentials \cite[\S~7]{derdzinski-maschler-03}, 
\cite[\S\S5--6]{derdzinski-maschler-06}; and functions on complex projective 
spaces obtained as ratios of suitable real quadratic forms 
(Example~\ref{fubst}).

This paper presents a classification of all triples $\,(M,g,\vp)\,$ formed by 
a compact K\"ah\-ler surface $\,(M,g)\,$ and a nonconstant Kil\-ling 
potential $\,\vp:M\to\bbR\,$ with a geodesic gradient. For those 
$\,(M,g,\vp)\,$ in which $\,\vp\,$ is not a special K\"ah\-\hbox{ler\hs-}\hskip0ptRic\-ci 
potential, $\,M\,$ must be a hol\-o\-mor\-phic $\,\bbCP^1$ bundle over a 
Riemann surface $\,\sd$, while $\,g\,$ and $\,\vp\,$ are obtained, via an 
explicit Ca\-la\-bi-style construction, from a Riemannian metric $\,h\,$ on 
$\,\sd$, a function $\,Q\,$ on a closed interval $\,\bbI$, subject only to 
specific positivity and boundary conditions, and a nonconstant mapping 
$\,\yj:\sd\to\bbRP^1\nh\smallsetminus\bbI\,$ (where 
$\,\bbI\subset\bbR\subset\bbRP^1$). The objects $\,\sd,h,\bbI,Q\,$ and 
$\,\yj$, being geometric invariants of the triple $\,(M,g,\vp)$, may be used 
to pa\-ram\-e\-trize the moduli space of such $\,(M,g,\vp)$.

Since special K\"ah\-\hbox{ler\hs-}\hskip0ptRic\-ci potentials on compact 
K\"ah\-ler manifolds have already been classified 
\cite{derdzinski-maschler-06}, the result just mentioned leads to a 
description of all compact K\"ah\-ler surfaces admitting nonconstant Kil\-ling 
potentials with geodesic gradients. They are bi\-hol\-o\-mor\-phic to total 
spaces of $\,\bbCP^1$ bundles, or to $\,\bbCP^2\nnh$. See 
\cite[\S\S5--6]{derdzinski-maschler-06}.

\section{Preliminaries}\label{pr}
All manifolds, mappings and tensor fields, including Riemannian metrics and 
functions, are assumed to be of class $\,C^\infty\nnh$. A (sub)manifold is by 
definition connected.

Let $\,\mathrm{Ric}\,$ be the Ric\-ci tensor of a tor\-sion-free connection 
$\,\nabla\,$ on a manifold $\,M$. Any vector field $\,v\,$ on $\,M\,$ 
satisfies the Boch\-ner identity
\begin{equation}\label{bch}
d\,\mathrm{div}\hskip2ptv\,=\,\mathrm{div}\hskip2pt\nabla\nh v\,
-\,\mathrm{Ric}\hh(\,\cdot\,,v)\hh,
\end{equation}
the coordinate form of which, 
$\,v^{\hs k}{}_{,\hs kj}=v^{\hs k}{}_{,jk}-R_{jk}v^{\hs k}$, arises by 
contraction in $\,l=k$ from the Ric\-ci identity $\,v^{\hs l}{}_{,jk}
-v^{\hs l}{}_{,\hs kj}=R_{jks}{}^lv^s\nnh$, which in turn is nothing else 
than the definition of the curvature tensor $\,R$. For such $\,M,\nabla\,$ and 
$\,v$, we treat $\,\nabla\nh v\,$ as the en\-do\-mor\-phism of the tangent 
bundle acting on vector fields $\,w\,$ by $\,w\mapsto\nabla_{\!v}w$, and then 
$\,\mathrm{div}\hskip2ptv=\mathrm{tr}\hskip2.5pt\nabla\nh v$.

Whenever $\,(M,g)\,$ is a Riemannian manifold, the symbol $\,\nabla\,$ will 
denote both the Le\-vi-Ci\-vi\-ta connection of $\,g\,$ and the 
$\,g$-gra\-di\-ent. If $\,\vp:M\to\bbR$, we have
\begin{equation}\label{tnd}
2\hh\nabla d\vp(v,\,\cdot\,)\,=\,dQ\hh,\hskip12pt\mathrm{where}\hskip7ptv
=\navp\hskip7pt\mathrm{and}\hskip7ptQ=g(v,v)\hs,
\end{equation}
as one sees noting that, in local coordinates, 
$\,(\vp_{,\hs k}\vp^{\hh,\hs k})_{,\hh j}^{\phantom i}
=2\vp_{,\hh kj}\vp^{\hh,\hs k}$.

Given a sub\-man\-i\-fold $\,\sd\,$ of a Riemannian manifold $\,(M,g)\,$
and $\,\ve\in(0,\infty)$, we denote by $\,N\sd\,$ the normal bundle of 
$\,\sd$, by $\,N^{\hh\ve}\nh\sd\,$ the (disjoint) union of radius $\,\ve$ 
open balls around $\,0\,$ in the normal spaces of $\,\sd$, by 
$\,B_\ve\nh(\nh\sd)\,$ the set of points of $\,M$ lying at distances less 
than $\,\ve\,$ from $\,\sd$, also called the 
$\,\ve${\it-neigh\-bor\-hood\/} of $\,\sd\,$ in $\,(M,g)$, by 
$\,\mathcal{D}\subset\tm\,$ is the domain of the exponential mapping 
$\,\mathrm{Exp}\,$ of $\,(M,g)$, and by 
$\,\mathrm{Exp}^\perp\nnh:\mathcal{D}\hh\cap N\sd\to M\,$ the {\it normal 
exponential mapping\/} of $\,\sd$, that is, the restriction of 
$\,\mathrm{Exp}\,$ to $\,\mathcal{D}\cap N\sd$. Thus, 
$\,N^{\hh\ve}\nh\sd\subset N\sd\,$ and $\,B_\ve\nh(\nh\sd)\subset M\,$ are 
open sub\-man\-i\-folds.
\begin{remark}\label{kilxp}As shown by Kobayashi \cite{kobayashi}, if $\,u\,$ 
is a Kil\-ling vector field on a Riemannian manifold $\,(M,g)$, the connected 
components of the zero set of $\,u\,$ are mutually isolated totally geodesic 
sub\-man\-i\-folds of even co\-di\-men\-sions. Every point of any such 
component $\,\sd$ \hbox{obviously} has a neighborhood $\,\sd'$ in $\,\sd\,$ 
with the property that, for some $\,\ve\in(0,\infty)$, the domain of 
$\,\mathrm{Exp}^\perp$ contains $\,N^{\hh\ve}\nh\sd'$ and 
$\,\mathrm{Exp}^\perp$ maps $\,N^{\hh\ve}\nh\sd'$ dif\-feo\-mor\-phic\-al\-ly 
onto an open set $\,\,U\subset M$. Whenever $\,\sd'\nnh,\ve\,$ and $\,\,U\,$ 
are chosen as above, the inverse of the dif\-feo\-mor\-phism 
$\,\mathrm{Exp}^\perp$ sends $\,u\,$ restricted to $\,\,U\,$ to a vector field 
$\,\hat u\,$ on $\,N^{\hh\ve}\nh\sd'$ which is vertical (tangent to the 
open-ball fibres $\,N_y^{\hh\ve}\nh\sd$, $\,y\in\sd'$) and, in each fibre 
$\,N_y^{\hh\ve}\nh\sd$, coincides with the linear vector field provided by the 
en\-do\-mor\-phism $\,[\nabla\nh u]_y$ of $\,T\hskip-2pt_y\hskip-.9ptM\,$ 
restricted to $\,N_y\sd$.

This is immediate since $\,\mathrm{Exp}^\perp$ maps short line segments 
emanating from $\,0\,$ in $\,N_y^{\hh\ve}\nh\sd\,$ onto geodesics, and so the 
local flow of $\,u\,$ in the sub\-man\-i\-fold 
$\,\mathrm{Exp}^\perp\nh(N_y^{\hh\ve}\nh\sd)$ corresponds, via 
$\,\mathrm{Exp}^\perp\nnh$, to the linear local flow near $\,0\,$ in 
$\,N_y\sd\,$ generated by $\,[\nabla\nh u]_y$.
\end{remark}
\begin{remark}\label{gauss}Let $\,\sd\,$ be a compact sub\-man\-i\-fold of 
a Riemannian manifold $\,(M,g)$. If $\,\ve\in(0,\infty)\,$ is 
sufficiently small, then the domain of $\,\mathrm{Exp}^\perp$ contains 
$\,N^{\hh\ve}\nh\sd\,$ and $\,\mathrm{Exp}^\perp$ maps $\,N^{\hh\ve}\nh\sd\,$ 
dif\-feo\-mor\-phic\-al\-ly onto $\,B_\ve\nh(\nh\sd)$. For any such 
$\,\ve$, the squared distance from $\,\sd\,$ is a $\,C^\infty$ function on 
$\,B_\ve\nh(\nh\sd)$, corresponding under the dif\-feo\-mor\-phism 
$\,\mathrm{Exp}^\perp$ to the squared-norm function on $\,N^{\hh\ve}\nh\sd$, 
and its $\,g$-gra\-di\-ent is tangent to all normal geodesics of lengths 
less that $\,\ve\,$ emanating from $\,\sd$, all of which are 
dis\-tance-min\-i\-miz\-ing.

The last claim follows from the generalized Gauss lemma, cf.\ 
\cite[p.\ 26]{gray}, in exactly the same way as the ordinary Gauss lemma is 
used to establish a special case of this claim, in which $\,\sd\,$ consists of 
a single point.
\end{remark}
The following well-known fact will be needed at the very end of 
Section~\ref{sp}.
\begin{lemma}\label{ismex}Let\/ $\,(\hat M,\hat g)\,$ and\/ $\,(M,g)\,$ be 
complete Riemannian manifolds with open subsets\/ 
$\,\hat M'\subset\hat M\,$ and\/ $\,M'\subset M\,$ such that both\/ 
$\,\hat M\smallsetminus\hat M'$ and\/ $\,M\smallsetminus M'$ are unions of 
finitely many compact submanifolds of codimensions greater than one. Any 
isometry of\/ $\,(\hat M'\nnh,\hat g)$ onto\/ $\,(M'\nnh,g)\,$ can then 
be uniquely extended to an isometry of\/ $\,(\hat M,\hat g)\,$ onto\/ 
$\,(M,g)$. If, in addition, $\,(\hat M,\hat g)\,$ and\/ $\,(M,g)\,$ are 
K\"ah\-ler manifolds and the isometry\/ $\,\hat M'\nh\to M'$ is a 
bi\-hol\-o\-mor\-phism, then so is the extension\/ $\,\hat M\to M$.
\end{lemma}
\begin{proof}See, for instance, \cite[Lemma~16.1]{derdzinski-maschler-06}.
\end{proof}
\begin{remark}\label{compl}We will use the eas\-i\-ly-ver\-i\-fied fact that a 
Riemannian manifold $\,(M,g)$ is complete if and only if every curve 
$\,(b,c)\ni t\mapsto x(t)\in M\,$ of finite length has limits as 
$\,t\to b\,$ and $\,t\to c$.
\end{remark}
\begin{remark}\label{rpone}We treat $\,\bbR\,$ as a subset of $\,\bbRP^1$ via 
the usual embedding $\,\vp\mapsto[\vp,1]\,$ (in homogeneous coordinates). For 
algebraic operations involving $\,\infty=[\hh1,0\hs]\in\bbRP^1$ and elements 
of $\,\bbR\subset\bbRP^1\nnh$, the standard conventions apply; thus, 
$\,p/\infty=0\,$ and $\,q/0=p+\infty=\infty\,$ if $\,p\in\bbR\,$ and 
$\,q\in\bbR\smallsetminus\{0\}$.
\end{remark}

\section{Kil\-ling Potentials}\label{kp}
The symbols $\,J\,$ and $\,\omega\,$ always stand for the 
com\-plex-struc\-ture tensor of a given K\"ah\-ler manifold $\,(M,g)\,$ and 
for its K\"ah\-ler form, with $\,\omega=g(J\,\cdot\,,\,\cdot\,)$. 
Real-hol\-o\-mor\-phic vector fields on $\,M\,$ then are the sections $\,v\,$ 
of $\,\tm\,$ such that $\,\Lie\hskip-.5pt_vJ=0$, which is equivalent to 
$\,[J,\nabla\nh v]=0$, the commutator $\,[\hskip2.5pt,\hskip1pt]\,$ being 
applied here to vec\-tor-bun\-dle morphisms $\,\tm\to\tm$. See, for instance, 
\cite[\S~5]{derdzinski-maschler-03}.

A $\,C^\infty$ function $\,\vp\,$ on a K\"ah\-ler manifold is a Kil\-ling 
potential (Section~\ref{in}) if and only if $\,v=\navp\,$ is a 
real-hol\-o\-mor\-phic vector field, cf.\ 
\cite[Lemma 5.2]{derdzinski-maschler-03}. In this case,
\begin{equation}\label{dvd}
d_v\hs\Delta\vp\hs\,=\,\hs2\hskip2pt\mathrm{div}\hskip2pt\nabla_{\!v}v\hs\,
-\,\hs2\hs|\nabla\nh v|^2,\hskip12pt\mathrm{where}\hskip7ptv=\navp\hs.
\end{equation}
In fact, the Boch\-ner identity (\ref{bch}) with $\,v=\navp\,$ reads 
$\,d\hs\Delta\vp=\mathrm{div}\hskip2pt\nabla d\vp
-\mathrm{Ric}\hh(\,\cdot\,,v)$. Multiplying both sides by $\,2\,$ and then 
subtracting the well-known equality
\begin{equation}\label{ddt}
d\hs\Delta\vp\,=\,-2\hs\mathrm{Ric}\hh(\,\cdot\,,v)\hh,\hskip12pt\mathrm{with}
\hskip7ptv=\navp\hs,
\end{equation}
valid whenever $\,\vp\,$ is a Kil\-ling potential \cite{calabi}, 
cf.\ \cite[formula (5.4)]{derdzinski-maschler-03}, we obtain 
$\,d\hs\Delta\vp=2\hskip2pt\mathrm{div}\hskip2pt\nabla d\vp$. Hence 
$\,d_v\hs\Delta\vp=2v^{\hs k}{}_{,jk}v^{\hs j}
=2(v^{\hs k}{}_{,\hh j}v^{\hs j}){}_{,\hs k}
-2v^{\hs k}{}_{,\hh j}v^{\hs j}{}_{,\hs k}$, as required.
\begin{remark}\label{rhcom}Given a Kil\-ling potential $\,\vp\,$ on a 
K\"ah\-ler manifold $\,(M,g)$, let us consider the vector fields 
$\,v=\navp\,$ and $\,u=Jv$. Then
\begin{enumerate}
  \def\theenumi{{\rm\alph{enumi}}}
\item $v,u\,$ are both real-hol\-o\-mor\-phic, and commute,
\item $u\,$ is a Kil\-ling field.
\end{enumerate}
Specifically, (b) amounts to the definition of a Kil\-ling potential at 
the beginning of Section~\ref{in}, and (a) is well known 
\cite[formula (5.1.b) and Lemma 5.2]{derdzinski-maschler-03}.
\end{remark}
A {\it special K\"ah\-ler-Ric\-ci potential\/} 
\cite[\S~7]{derdzinski-maschler-03}. on a K\"ah\-ler manifold $\,(M,g)\,$ is 
any nonconstant Kil\-ling potential $\,\vp\,$ such that, at points where 
$\,d\vp\ne0$, all nonzero vectors orthogonal to $\,\navp\,$ and $\,J(\navp)\,$ 
are eigen\-vec\-tors of both $\,\nabla d\vp\,$ and $\,\mathrm{Ric}$. 
\begin{remark}\label{cifot}Let $\,\vp\,$ and $\,f\,$ be functions on a 
manifold $\,M\,$ such that $\,\vp\,$ is nonconstant and $\,f=\chi\circ\vp\,$ 
with some $\,C^\infty$ function $\,\chi:\bbI\to\bbR$, where $\,\bbI=\vp(M)\,$ 
is the range of $\,\vp$. We then say that $\,f\,$ is {\it a\/ $\,C^\infty$ 
function of\/} $\,\vp$.
\end{remark}
\begin{remark}\label{skrsu}
In view of (\ref{tnd}) and (\ref{ddt}), a nonconstant Kil\-ling potential 
$\,\vp\,$ on a K\"ah\-ler surface $\,(M,g)\,$ is a special 
K\"ah\-\hbox{ler\hs-}\hskip0ptRic\-ci potential if and only if every point 
with $\,d\vp\ne0$ has a neighborhood on which both $\,Q=g(\navp,\navp)\,$ 
and $\,\Delta\vp\,$ are $\,C^\infty$ functions of $\,\vp$.
\end{remark}

\section{Geodesic Gradients: the Simplest Examples}\label{se}
Let $\,\nabla\,$ be a connection in the tangent bundle $\,\tm\,$ of a manifold 
$\,M\nh$. A {\it geodesic vector field\/} relative to $\,\nabla\nh\,$ is any 
vector field $\,v\,$ on $\,M\,$ such that, for some function 
$\,\ta:M'\nh\to\bbR\,$ defined on the open set $\,M'\nh\subset M\,$ on which 
$\,v\ne0$,
\begin{equation}\label{nvv}
\nabla_{\!v}v\,=\,\ta\hskip.4ptv\qquad\mathrm{everywhere\ in}\hskip7ptM'\nh,
\end{equation}
or, equivalently, such that the integral curves of $\,v\,$ are 
re\-pa\-ram\-e\-trized $\,\nabla\nnh$-ge\-o\-des\-ics.

We say that a function $\,\vp:M\to\bbR\,$ on a Riemannian manifold $\,(M,g)\,$ 
has a {\it geodesic gradient\/} if $\,v=\navp\,$ is a geodesic vector field 
for the Le\-vi-Ci\-vi\-ta connection $\,\nabla\,$ of $\,g$. It is clear from 
(\ref{tnd}) and (\ref{nvv}) that this amounts to the condition
\begin{equation}\label{qeg}
dQ\wedge\hs d\vp\,=\,0\hs,\qquad\mathrm{where}\quad Q\,
=\,g(\navp,\navp)\hs,
\end{equation}
which is in turn the same as requiring $\,Q\,$ to be, locally in $\,M'\nnh$, a 
function of $\,\vp$.
\begin{remark}\label{ggtfc}If $\,v\,$ is a geodesic vector field for a 
connection $\,\nabla\,$ on $\,M$, then so is $\,\mu\hh v$ for any function 
$\,\mu:M\to\bbR$.
\end{remark}
\begin{example}\label{geogr}Each of the following assumptions about a given 
Riemannian manifold $\,(M,g)\,$ and a function $\,\vp:M\to\bbR\,$ implies that 
$\,\vp\,$ has a geodesic gradient.
\begin{enumerate}
  \def\theenumi{{\rm\alph{enumi}}}
\item Some group of isometries of $\,(M,g)\,$ with principal orbits of 
co\-di\-men\-sion $\,1$ leaves $\vp\,$ invariant.
\item $\dim M=1$.
\item $\,\vp=\chi\circ\rho\,$ for some function $\,\rho\,$ on $\,(M,g)\,$ 
that has a geodesic gradient and some $\,\chi:\bbI\to\bbR$, where 
$\,\bbI\subset\bbR\,$ is an interval containing the range $\,\rho(M)$.
\item $(M,g)\,$ is the $\,\ve$-neigh\-bor\-hood, for any sufficiently small 
$\,\ve\in(0,\infty)$, of a given compact sub\-man\-i\-fold $\,\sd\,$ in 
a Riemannian manifold, and $\,\vp\,$ is the squared distance from 
$\,\sd$.
\item $(M,g)\,$ is a Riemannian product and $\,\vp\,$ is a function with a 
geodesic gradient on one of the factor Riemannian manifolds, treated as a 
function on $\,M$.
\end{enumerate}
For (a) this is a direct consequence of (\ref{qeg}), as the gradients of 
$\,\vp\,$ and $\,Q\,$ are both normal to the orbits; (b) leads to (a) for the 
trivial group; and the claims in (c) -- (d) easily follow from 
Remarks~\ref{ggtfc} and~\ref{gauss}, while the case of (e) is obvious.
\end{example}
\begin{example}\label{every}A a nonconstant function $\,\vp\,$ with a geodesic 
gradient exists on every Riemannian manifold $\,(M,g)$, and may be chosen so 
that $\,0\,$ is a regular value of $\,\vp$, and $\,\vp^{-1}(0)\,$ is any 
prescribed compact sub\-man\-i\-fold $\,\sd\,$ of co\-di\-men\-sion $\,1\,$ 
which disconnects $\,M\,$ (such as a sphere embedded in a coordinate 
domain).

In fact, for $\,\ve\,$ as in Remark~\ref{gauss} and a unit normal vector 
field $\,w\,$ along $\,\sd$, the assignment 
$\,(y,t)\mapsto\mathrm{exp}_y\hs tw_y$ defines a dif\-feo\-mor\-phism 
$\,\sd\times(\nh-\ve,\ve)\to B_\ve\nh(\nh\sd)$. As the function 
$\,\rho:B_\ve\nh(\nh\sd)\to\bbR\,$ sending $\,\mathrm{exp}_y\hs tw_y$ to 
$\,t\,$ has a geodesic gradient (cf.\ Remark~\ref{gauss}), we may set 
$\,\vp=\chi\circ\rho$, as in (iii), with $\,\chi:\bbR\to\bbR\,$ that is 
nondecreasing, constant on both $\,(\nh-\infty,-\delta)\,$ and 
$\,(\delta,\infty)\,$ for some $\,\delta\in(0,\ve)$, and equal to the 
identity on a neighborhood of $\,0$.
\end{example}
\begin{example}\label{skrgg}Every special K\"ah\-\hbox{ler\hs-}\hskip0ptRic\-ci potential on a 
K\"ah\-ler manifold (Section~\ref{kp}) has a geodesic gradient, which is 
immediate as (\ref{tnd}) then implies (\ref{qeg}).
\end{example}
\begin{example}\label{fubst}For fixed nonnegative integers $\,k,l,m\,$ with 
$\,m=k+l+1\ge2$, let $\,g\,$ be the Fu\-bi\-ni-Stu\-dy metric on 
$\,M=\bbCP^{\hs m}\nnh$. Then $\,\vp:M\to\bbR\,$ defined by the assignment 
$\,[x,y]\mapsto|y|^2\nh/(|x|^2\nh+|y|^2)$, where $\,[x,y]\,$ are the 
homogeneous coordinates, while $\,x\in\bbC^{\hh k+1}$ and 
$\,y\in\bbC^{\hs l+1}\nnh$, is a nonconstant Kil\-ling potential with a 
geodesic gradient. More precisely, it is easy to verify that $\,Q\,$ in 
(\ref{qeg}) equals $\,4(1-\vp)\vp$, so that the critical points of 
$\,\vp\,$ form the union of two disjoint linear varieties $\,\bbCP^{\hh k}$ 
and $\,\bbCP^{\hs l}$ in $\,\bbCP^{\hs m}\nnh$.
\end{example}
\begin{remark}\label{lngth}Let $\,\vp\,$ be a function with a geodesic 
gradient exists on a Riemannian manifold. For any nonconstant integral curve 
$\,t\mapsto x(t)\,$ of the gradient $\,v=\navp$, the $\,\vp$-im\-age of the 
curve has the form $\,(b,c)$, with $\,-\infty\le b<c\le\infty$. Since 
$\,\vp\,$ is an increasing function of $\,t$, it can be used as a new curve 
parameter. In terms of $\,\vp$, the length of the curve obviously equals 
$\,\int_b^{\hs c}Q^{-1/2}\,d\vp$, where $\,Q=g(v,v)$.
\end{remark}

\section{Further Examples and a Classification Theorem}\label{fe}
The following construction generalizes that of 
\cite[\S5]{derdzinski-maschler-06} (in the case $\,m=2$), and gives rise to 
compact K\"ah\-ler surfaces $\,(M,g)\,$ with nonconstant Kil\-ling potentials 
$\,\vp$, which have geodesic gradients, but, in contrast with 
\cite[\S5]{derdzinski-maschler-06}, need {\it not\/} be special 
K\"ah\-\hbox{ler\hs-}\hskip0ptRic\-ci potentials. For a detailed comparison 
with \cite[\S5]{derdzinski-maschler-06}, see Remark~\ref{cmpre} below.

One begins by fixing a nonuple
\begin{equation}\label{dat}
\bbI,\,a,\,\sd,\,h,\,\mathcal{L},\,(\hskip2.2pt,\hskip1pt),\,\mathcal{H},\,\yj,\,Q
\end{equation}
consisting of the following objects:
\begin{enumerate}
  \def\theenumi{{\rm\roman{enumi}}}
\item a nontrivial closed interval 
$\,\bbI=[\hh\vp_{\mathrm{min}},\vp_{\mathrm{max}}]\,$ of the variable $\,\vp$,
\item a real number $\,a>0$,
\item a compact K\"ah\-ler manifold $\,(\sd,h)\,$ of complex dimension $\,1$, 
\item a $\,C^\infty$ function $\,Q:\bbI\to\bbR\,$ equal to $\,0\,$ at the 
endpoints of $\,\bbI$, positive on its interior $\,\bbI^\circ\nnh$, with 
$\,dQ\hh/\nh d\vp=2a\,$ at $\,\vp_{\mathrm{min}}$ and 
$\,dQ\hh/\nh d\vp=-2a\,$ 
at $\,\vp_{\mathrm{max}}$,
\item a $\,C^\infty$ mapping $\,\yj:\sd\to\bbRP^1\nh\smallsetminus\bbI$, with 
$\,\bbI\subset\bbR\subset\bbRP^1$ as in Remark~\ref{rpone},
\item a $C^\infty$ complex line bundle $\,\mathcal{L}\,$ over $\,\sd\,$ with a 
Hermitian fibre metric $\,(\hskip2.2pt,\hskip1pt)$,
\item the horizontal distribution $\,\mathcal{H}\,$ of a connection in 
$\,\mathcal{L}\,$ making $\,(\hskip2.2pt,\hskip1pt)\,$ parallel and having the 
curvature form $\,\varOmega\hs=-\hs a\hs(\vp_*-\yj)^{-1}\omega^{(h)}\nh$,
\end{enumerate}
where $\,\omega^{(h)}$ is the K\"ahler form of $\,(\sd,h)$. Thus, 
$\,\varOmega=0\,$ at points at which $\,\yj=\infty$. Note that, in (iii), 
$\,(\sd,h)\,$ is nothing else than a closed oriented real surface endowed with 
a Riemannian metric. 

In addition to the data (\ref{dat}), let us fix a $\,C^\infty$ 
dif\-feo\-mor\-phism $\,\bbI^\circ\nh\ni\vp\mapsto r\in(0,\infty)$ such that 
$\,dr/d\vp=ar/Q$, and a ``base point'' $\,\vp_*\in\bbI$. We choose $\,\vp_*$ 
to be the midpoint of $\,\bbI$, which is just an arbitrary normalization. 
See Remark~\ref{bspnt}.

We use the symbol $\,\mathcal{V}\,$ for the vertical distribution 
$\,\mathrm{Ker}\hskip2.7ptd\prj\,$ on the total space of the bundle (also 
denoted by $\,\mathcal{L}$), $\,\prj:\mathcal{L}\to\sd\,$ being the bundle 
projection. From now on the norm function $\,r:\mathcal{L}\to[\hs0,\infty)\,$ 
of $\,(\hskip2.2pt,\hskip1pt)\,$ is treated, simultaneously, as an independent 
variable ranging over $\,[\hs0,\infty)$, so that our fixed 
dif\-feo\-mor\-phism $\,\vp\mapsto r$ turns $\,\vp$, and hence $\,Q\,$ as 
well, into functions $\,\mathcal{L}\to\bbR$.

Next we define a Riemannian metric $\,g\,$ on 
$\,M'\nh=\mathcal{L}\smallsetminus\sd$, where $\,\sd\,$ is identified with 
the zero section, by $\,g=(\vp_*-\yj)^{-1}(\vp-\yj)\hs h\,$ or $\,g=h\,$ on 
$\,\mathcal{H}$, 
$\,g=(ar)^{-2}Q\,\mathrm{Re}\hskip1pt(\hskip2.2pt,\hskip1pt)\,$ on 
$\,\mathcal{V}\nnh$, and $\,g(\mathcal{H},\mathcal{V})=\{0\}$. Tensors on 
$\,\sd\,$ are denoted by the same symbols as their pull\-backs to $\,M'\nnh$, 
so that $\,\yj\,$ stands here for $\,\yj\circ\prj\,$ and $\,h\,$ for 
$\,\prj^*\nh h$. On $\,\mathcal{H}$, the first formula is to be used in the 
$\,\prj$-pre\-im\-age of the set in $\,\sd\,$ on which $\,\yj\ne\infty$, and 
the second one on its complement. Note that 
$\,C^\infty\nnh$-dif\-fer\-en\-tia\-bil\-i\-ty of the algebraic operations in 
$\,\bbRP^1\nnh$, wherever they are permitted (cf.\ Remark~\ref{rpone}) implies 
that $\,g\,$ is of class $\,C^\infty\nnh$.

Obviously, $\,(M'\nnh,g)\,$ is an almost Her\-mit\-i\-an manifold for the 
almost complex structure $\,J\,$ obtained by requiring that the sub\-bun\-dles 
$\,\mathcal{V}\,$ and $\,\mathcal{H}\,$ of $\,\tm'$ be $\,J$-in\-var\-i\-ant 
and, for any $\,x\in M'\nnh$, the restriction of $\,J_x$ to 
$\,\mathcal{V}\nh_x$, or $\,\mathcal{H}_x$, coincide with the complex 
structure of the fibre $\,\mathcal{L}_{\prj(x)}$ or, respectively, with the 
$\,d\prj_x$-pull\-back of the complex structure of $\,\sd$.

Let $\,M\,$ be the $\,\bbCP^1$ bundle over $\,\sd\,$ resulting from the 
projective compactification of $\,\mathcal{L}$. Our $\,g,\vp\,$ and $\,J\,$ 
then have $\,C^\infty$ extensions to a metric, function, and almost complex 
structure on $\,M\,$ denoted, again, by $\,g,\vp\,$ and $\,J$. In fact, such 
extensions exist for the distributions $\,\mathcal{V}\,$ and $\,\mathcal{H}$. 
Our claim thus follows since, according to the conclusion made in 
\cite[\S5]{derdzinski-maschler-06} for $\,m=1$, the function $\,\vp\,$ 
restricted to the subset $\,\mathcal{L}_y\smallsetminus\{0\}\,$ of a single 
fibre of $\,\mathcal{L}_y$ of $\,\mathcal{L}$, and the metric 
$\,(ar)^{-2}Q\,\mathrm{Re}\hskip1pt(\hskip2.2pt,\hskip1pt)\,$ on 
$\,\mathcal{L}_y\smallsetminus\{0\}$, can both be smoothly extended to the 
Riem\-ann-sphere compactification of $\,\mathcal{L}_y$. 

For the section $\,v\,$ of the vertical distribution $\,\mathcal{V}\,$ on 
$\,\mathcal{L}\,$ which, restricted to each fibre of $\,\mathcal{L}$, 
equals $\,a\,$ times the radial (identity) vector field on the fibre, one 
easily verifies that $\,d_v=Q\,d/d\vp$, both sides being viewed as operators 
acting on $\,C^\infty$ functions of $\,\vp$. Consequently, $\,v\,$ equals the 
$\,g$-gradient $\,\navp\,$ of $\,\vp$. Note that $\,g(v,v)=Q$.

From now on the symbols $\,w,w\hh'$ will stand both for any two $\,C^\infty$ 
vector fields in $\,\sd\,$ and, simultaneously, for their horizontal lifts to 
$\,\mathcal{L}\,$ (which themselves are just the $\,\prj$-pro\-jecta\-ble 
horizontal vector fields on $\,\mathcal{L}$). We also define a vector field 
$\,u\,$ on $\,\mathcal{L}$ by $\,u=iv\,$ (multiplication by $\,i\,$ in each 
fibre), so that, for our $\,J$, and $\,w\,$ as above, $\,Jv=u$, while $\,Jw\,$ 
has the same meaning in $\,\mathcal{L}\,$ as in $\,\sd$. With  $\,\nabla\,$ 
and $\,\mathrm{D}\,$ denoting the Le\-vi-Ci\-vi\-ta connections of $\,g\,$ and 
$\,h$, one has, on a dense open subset of $\,M'\nnh$,
\begin{equation}\label{nab}
\begin{array}{l}
\nabla_{\!v}v\,=\,-\hs\nabla_{\!u}u\,=\,\ta\hskip.4ptv\hs,\quad
\nabla_{\!v}\hs u\,=\,\nabla_{\!u}v\,=\,\ta\hskip.4ptu\hs,\\
\nabla_{\!v}w\,=\,\nabla_{\!w}v\,=\,\si\hskip.4ptw\hs,\quad
\nabla_{\!u}w\,=\,\nabla_{\!w}\hs u\,=\,\si\hskip.4ptJw\hs,\\
Q\nabla_{\!w}w\hh'\,=\,\,Q\hs\mathrm{D}_ww\hh'\,
-\,\si\hs[g(w,w\hh')v\hs\,+\,\hs g(Jw,w\hh')u]\\
\phantom{\nabla_{\!w}w\hh'\,\,}
+\,\,\hh(\vp_*-\yj)^{-1}(\vp-\vp_*)\si\hs
[\hh h(\mathrm{D}\hh\yj\hh,w)\hh w\hh'\nh+h(\mathrm{D}\hh\yj\hh,w\hh'\hh)\hh w
-h(w,w\hh'\hh)\hs\mathrm{D}\hh\yj]
\end{array}
\end{equation}
for $\,\ta,\si:M'\to\bbR\,$ given by $\,2\hh\ta=dQ\hh/\nh d\vp\,$ and 
$\,2\hh\si=(\vp-\yj)^{-1}Q$. The dense open set in question is the union of 
the $\,\prj$-pre\-im\-ages of two subsets in $\,\sd$, which are: the 
$\,\yj$-pre\-im\-age of $\,\bbR=\bbRP^1\smallsetminus\{\infty\}$, 
cf.\ Remark~\ref{rpone}; and the interior of the $\,\yj$-pre\-im\-age of 
$\,\infty$. On the former set, $\,\mathrm{D}\hh\yj\,$ denotes the 
$\,h$-gradient of $\,\yj\,$ treated as a real-val\-ued function; on the 
latter, we set $\,\mathrm{D}\hh\yj=0$.

In fact, the connection $\,\nabla\,$ {\it defined\/} by (\ref{nab}) is clearly 
compatible with $\,g\,$ and tor\-sion-free, since $\,v,u\,$ commute both with 
each other and with the horizontal lifts $\,w,w\hh'\nnh$, while the vertical 
component of $\,[w,w\hh'\hh]\,$ is $\,a^{-1}\varOmega(w,w\hh'\hh)\hs u$, cf.\ 
\cite[formula (3.6)]{derdzinski-maschler-03}.

Also, $\,J\,$ commutes with $\,\nabla_{\!v},\hs\nabla_{\!u}$, all 
$\,\nabla_{\!w}$, and $\,\nabla\nh v$. These commutation relations are 
obvious from (\ref{nab}), possibly except $\,[J,\nabla_{\!w}]\hh w\hh'\nh=0$, 
which follows, as (\ref{nab}) yields
\[
[J,\nabla_{\!w}]\hh w\hh'=\hs[(\vp_*-\yj)Q]^{-1}(\vp-\vp_*)\si
\hskip2pt[\varXi(Jw,w\hh'\nnh,J\hh\mathrm{D}\hh\yj)-
\varXi(w,w\hh'\nnh,\mathrm{D}\hh\yj)]\hh,
\]
with $\,\varXi(w,w\hh'\nnh,w\hh''\hh)=h(Jw,w\hh'\hh)\hh w\hh''\nh
+h(Jw\hh'\nnh,w\hh''\hh)\hh w+h(Jw\hh''\nnh,w)\hh w\hh'\nnh$. 
Skew-sym\-me\-try of $\,\varXi\,$ and 
\hbox{two\hh-}\hskip0ptdi\-men\-sion\-al\-i\-ty of $\,\sd\,$ now give 
$\,\varXi(w,w\hh'\nnh,w\hh''\hh)=0$.

The conclusions of the last paragraph amount to $\,\nabla J=0\,$ and 
$\,[J,\nabla\nh v]=0$. The former equality means that $\,g\,$ is a K\"ahler 
metric; the latter states that $\,v=\navp\,$ is real-hol\-o\-mor\-phic, which 
makes $\,\vp\,$ a (nonconstant) Kil\-ling potential on the K\"ah\-ler manifold 
$\,(M,g)$, cf.\ Section~\ref{kp}. Also, $\,\vp\,$ has a geodesic gradient in 
view of the first line in (\ref{nab}). Note that 
$\,\Delta\vp=\mathrm{tr}\hskip2pt\nabla\nh v=2\hh\si+2\hh\ta$, and so
\begin{equation}\label{dte}
\Delta\vp\hs\,=\,\hs(\vp-\yj)^{-1}Q\hs\,+\hs\,dQ\hh/\nh d\vp\hh.
\end{equation}
\begin{remark}\label{cmpre}By (\ref{dte}) and Remark~\ref{skrsu}, our 
$\,\vp\,$ is a special K\"ah\-\hbox{ler\hs-}\hskip0ptRic\-ci potential on 
$\,(M,g)\,$ if and only if $\,\yj\,$ is constant. When $\,\yj\,$ is constant, 
our construction becomes that of \cite[\S5]{derdzinski-maschler-06} for 
$\,m=2,\vp_0^{\phantom i}=\vp_{\mathrm{min}}\hs$, and either $\,\ve=0\,$ with 
an undefined constant $\,c$ (when $\,\yj=\infty$), or $\,\ve=\pm1\,$ with 
$\,c\in\bbR=\bbRP^1\smallsetminus\{\infty\}\,$ equal to the value of 
$\,\yj\,$ (if $\,\yj\ne\infty$); in the latter case, our $\,h\,$ is 
$\,2\hs|\vp_*\nh-c\hh|\,$ times the metric denoted by $\,h\,$ in 
\cite{derdzinski-maschler-06}. 
\end{remark}
\begin{remark}\label{bspnt}The ``base point'' $\,\vp_*$ is not a geometric 
invariant of the triple $\,(M,g,\vp)$ constructed above, and one may choose it 
to be a different constant, or even a {\it function\/} 
$\,\widetilde{\vp}_*:\sd\to\bbR$, as long as $\,\vp\ne\yj\ne\widetilde{\vp}_*$ 
everywhere in $\,M$, so that the definition of $\,g\,$ makes sense. (Again, we 
treat $\,\vp,\yj\,$ and $\,\widetilde{\vp}_*$ as functions $\,M\to\bbR$.) The 
resulting metric $\,g\,$ will then remain unchanged, provided that we replace 
$\,h\,$ with $\,\widetilde h$, equal to 
$\,(\vp_*-\yj)^{-1}(\widetilde{\vp}_*-\yj)\hs h\,$ on the subset of $\,\sd\,$ 
on which $\,\yj\ne\infty$, and to $\,h\,$ on its complement. (Condition (vii) 
for $\,\widetilde{\vp}_*$ and $\,\widetilde h\,$ will still hold, with the 
same $\,\mathcal{H}\,$ and $\,\varOmega$.)

More generally, we can relax conditions (iii) -- (v), while keeping (ii), (vi) 
and (vii), so that $\,\sd\,$ need not be compact, $\,Q\,$ is 
defined and positive on an open interval, and $\,\yj,\vp_*:\sd\to\bbRP^1\nnh$. 
The construction then yields a triple $\,(M,g,\vp)\,$ with the same 
properties, except compactness of $\,M$, where $\,M\,$ now is any connected 
component of the open set in $\,\mathcal{L}\smallsetminus\sd\,$ defined by 
requiring that $\,\vp\ne\yj\ne\vp_*$ and that the values of the norm function 
$\,r\,$ lie in the resulting new range.
\end{remark}
Compact K\"ah\-ler manifolds of all dimensions, admitting special 
K\"ah\-\hbox{ler\hs-}\hskip0ptRic\-ci potentials, have been completely 
described in \cite[Theorem 16.3]{derdzinski-maschler-06}. Combined with the 
following result, this provides a classification of compact K\"ah\-ler 
surfaces with nonconstant Kil\-ling potentials that have geodesic gradients.
\begin{theorem}\label{clssf}Let\/ $\,\vp\,$ be a nonconstant Kil\-ling 
potential with a geodesic gradient on a compact K\"ah\-ler surface\/ 
$\,(M,g)$. If\/ $\,\vp\,$ is not a special K\"ah\-\hbox{ler\hs-}\hskip0ptRic\-ci potential on\/ 
$\,(M,g)$, then, up to a bi\-hol\-o\-mor\-phic isometry, the triple\/ 
$\,(M,g,\vp)\,$ arises from the above construction applied to some data\/ 
{\rm(\ref{dat})} satisfying conditions\/ {\rm(i)} -- {\rm(vii)}, such that\/ 
$\,\yj:\sd\to\bbRP^1\nh\smallsetminus\bbI\,$ is nonconstant.
\end{theorem}
A proof of Theorem~\ref{clssf} is given in Sections~\ref{pf} and~\ref{sp}.

\section{One-jets of Geodesic Vector Fields at Their Zeros}\label{gv}
As a first step toward the proof of Theorem~\ref{clssf}, we now proceed to 
establish one general property of geodesic vector fields, defined in 
Section~\ref{se}.
\begin{remark}\label{curve}If $\,\ve>0\,$ and a curve 
$\,[\hs0,\ve]\ni t\mapsto v(t)\in V\,$ in a normed vector space $\,V$ with 
$\,\,\dim V<\infty\,$ is differentiable at $\,t=0$, while 
$\,v(0)=0\ne w$, where $\,w=\dot v(0)$ and $\,\dot v=dv/dt$, then 
$\,v(t)/|v(t)|\,\to\,w/|w|\,$ as\/ $\,t\to0^+\nnh$. (In fact, 
$\,v(t)/t\to\dot v(0)=w$ as $\,t\to0^+\nnh$. Thus, $\,|v(t)|/t\to|w|\,$ and 
$\,v(t)/|v(t)|=[v(t)/t][|v(t)|/t]^{-1}\hh\to\,w/|w|$.)
\end{remark}
\begin{lemma}\label{diagz}Let\/ $\,v\,$ be a geodesic vector field on a 
manifold\/ $\,M\,$ with a fixed connection\/ $\,\nabla\nh$. If\/ $\,y\in M\,$ 
and\/ $\,v_y=0$, then, for\/ $\,\nv=[\nabla\nh v]_y:T_yM\to T_yM\,$ and 
some\/ $\,\lar\in\bbR$, we have\/ $\,\nv^2=\lar\nv$, that is, one of the 
following two cases occurs\/{\rm:}
\begin{enumerate}
  \def\theenumi{{\rm\roman{enumi}}}
\item $\nv\,$ is diagonalizable, and either it is a multiple of the 
identity, or it has exactly two distinct eigenvalues, one of which is zero.
\item $\nv\,$ is not diagonalizable and\/ $\,\nv^2\nh=0$. 
\end{enumerate}
\end{lemma}
\begin{proof}We may assume that $\,\nv\ne0\,$ and identify a neighborhood of 
$\,y\,$ in $\,M\,$ with a neighborhood 
$\,\,U\,$ of $\,0\,$ in a vector space $\,V\nnh$, so that $\,y\,$ corresponds 
to $\,0$. This turns $\,\nabla\,$ into a connection in $\,TU\nh$. As $\,v=0\,$ 
at the point $\,0$, the operator $\,\nv\,$ is now the differential at $\,0\,$ 
of $\,v\,$ viewed as a mapping $\,\,U\nh\to V\nnh$. We also fix a vector 
subspace $\,V'\nnh\subset V\,$ of dimension $\,\mathrm{rank}\,\nv\,$ such that 
$\,\nv\,$ maps $\,V'$ isomorphically onto the image $\,\nv(V)$, and choose a 
linear projection $\,\pro:V\nh\to\nv(V)$. In view of the inverse mapping 
theorem, there exists a neighborhood $\,\,U'$ of $\,0\,$ in $\,V'$ such that 
$\,\,U'\nnh\subset U\,$ and $\,\varPi=\,\pro\hh\circ\hs v:U'\to U''$ is a 
diffeomorphism onto a neighborhood $\,\,U''$ of $\,0\,$ in $\,\nv(V)$. Thus, 
$\,\varPi(0)=0\,$ and $\,d\varPi_0^{\phantom i}$ equals $\,\nv\,$ restricted 
to $\,V'\nnh$. 

Given any nonzero vector $\,w\in\nv(V)$, let $\,\ve>0\,$ be such that 
$\,tw\in U''$ for all $\,t\in[\hs0,\ve]$. We set $\,x(t)=\varPi^{-1}(tw)\,$ if 
$\,t\in[\hs0,\ve]$. Thus, $\,v(x(t))\ne0\,$ for $\,t\in(0,\ve]$, as 
$\,\pro\hs v(x(t))=\varPi(x(t))=tw\ne0$. We may now set 
$\,u(t)=v(x(t))/|v(x(t))|$, if $\,0<t\le\ve$, using a fixed norm 
$\,|\hskip3pt|\,$ in $\,V\nnh$, so that $\,u(t)\to w/|w|\,$ as $\,t\to0^+$ 
according to Remark~\ref{curve}, and an equality of the form (\ref{nvv}) holds 
at each $\,x(t)$, $\,t\in(0,\ve]$, with some function $\,\ta\,$ (defined only 
at points where $\,v\ne0$). Dividing both sides of that equality by 
$\,|v(x(t))|\,$ and setting $\,\lar(t)=\ta(x(t))$, we obtain 
$\,[\nabla_{\!u(t)}v]_{x(t)}\nh=\lar(t)u(t)$. Consequently, $\,\lar(t)\,$ has 
a limit $\,\lar_w^{\phantom i}$ as $\,t\to0^+\,$ and, taking the limits of 
both sides of the last relation, we get 
$\,[\nabla_{\!w}v]_0^{\phantom i}=\lar_w^{\phantom i}w$, that is, 
$\,\nv w=\lar_w^{\phantom i}w$. Every $\,w\in\nv(V)\smallsetminus\{0\}\,$ is 
thus an eigenvector of $\,\nv\,$ for some eigenvalue $\,\lar_w^{\phantom i}$, 
which is only possible if $\,\lar=\lar_w^{\phantom i}$ does not depend on 
$\,w$. Hence $\,\nv(V)\subset\hs\mathrm{Ker}\,(\nv-\lar)\,$ or, equivalently, 
$\,\nv^2-\lar\nv=(\nv-\lar)\nv=0$. If $\,\lar\ne0$, the subspaces 
$\,\mathrm{Ker}\,\nv\,$ and $\,\mathrm{Ker}\,(\nv-\lar)\,$ must, for 
dimensional reasons, be the summands in a direct-sum decomposition of 
$\,V\nh$. This leads to case (i). Hence, if $\,\nv\,$ is not diagonalizable, 
we have $\,\lar=0$, and (ii) follows.
\end{proof}

\section{Morse\hs-Bott Functions with Geodesic Gradients}\label{mb}
A {\it\hbox{Morse\hs-}\hskip0ptBott function\/} on a manifold $\,M\,$ is a 
$\,C^\infty$ function $\,\vp:M\to\bbR\,$ such that the connected components 
of the set of critical points of $\,\vp\,$ are mutually isolated 
sub\-man\-i\-folds of $\,M\,$ (called the {\it critical manifolds\/} of 
$\,\vp$), and the rank of the Hess\-i\-an of $\,\vp\,$ at every critical point 
$\,x\,$ is the co\-di\-men\-sion of the critical manifold containing $\,x$.
\begin{example}\label{kpamb}All Kil\-ling potentials are 
\hbox{Morse\hs-}\hskip0ptBott functions, and their critical manifolds are 
totally geodesic complex sub\-man\-i\-folds of the ambient K\"ah\-ler 
manifold. This is a well-known consequence of Remark~\ref{rhcom}(b) and 
Kobayashi's result \cite{kobayashi} mentioned in Remark~\ref{kilxp}. Cf.\ also 
\cite[Example 11.1 and Remark 2.3(iii-\hh c,\hs d)]{derdzinski-maschler-06}.
\end{example}
\begin{remark}\label{sqdst}The standard examples of 
\hbox{Morse\hs-}\hskip0ptBott functions are provided by homogeneous quadratic 
polynomials on fi\-\hbox{nite\hh-}\hskip0ptdi\-men\-sion\-al real vector 
spaces. The conclusion about the squared-norm function in Remark~\ref{gauss} 
now implies that $\,\vp\,$ of Example~\ref{geogr}(d) is a 
\hbox{Morse\hs-}\hskip0ptBott function.
\end{remark}
The next remark and lemma use the symbols $\,\mathrm{Exp}^\perp$ and 
$\,N^{\hh\ve}\nh\sd\,$ defined in Section~\ref{pr}.
\begin{remark}\label{invma}Given a critical manifold $\,\sd\,$ of a 
\hbox{Morse\hs-}\hskip0ptBott function $\,\vp\,$ on a manifold $\,M\,$ and a 
point $\,y\in\sd$, there exist a neighborhood $\,\sd'$ of $\,y\,$ in $\,\sd\,$ 
and $\,\ve\in(0,\infty)\,$ such that the domain of $\,\mathrm{Exp}^\perp$ 
contains $\,N^{\hh\ve}\nh\sd'$ and $\,\mathrm{Exp}^\perp$ maps 
$\,N^{\hh\ve}\nh\sd'$ dif\-feo\-mor\-phic\-al\-ly onto a neighborhood 
$\,\,U\,$ of $\,y\,$ in $\,M$, while $\,\navp\ne0\,$ everywhere in 
$\,\,U\smallsetminus\sd'\nnh$.

This is immediate from the inverse mapping theorem applied to 
$\,\mathrm{Exp}^\perp$ and the definition of a critical manifold.
\end{remark}
\begin{lemma}\label{mrbgg}Let\/ $\,y\in\sd$, for a critical manifold\/ 
$\,\sd\,$ of a nonconstant \hbox{Morse\hs-}\hskip0ptBott function\/ $\,\vp$ 
with a geodesic gradient on a Riemannian manifold\/ $\,(M,g)$.
\begin{enumerate}
  \def\theenumi{{\rm\roman{enumi}}}
\item The Hess\-i\-an\/ $\,\nabla d\vp\,$ at\/ $\,y\,$ has exactly one nonzero 
eigen\-value\/ $\,a$.
\item The eigen\-space corresponding to\/ $\,a\,$ in\/ {\rm(i)} is the normal 
space\/ $\,N_y\sd\,$ of\/ $\,\sd\,$ at\/ $\,y$.
\item For every sufficiently small\/ $\,\ve\in(0,\infty)\,$ there exists a 
neighborhood\/ $\,\,U\,$ of\/ $\,y\,$ in\/ $\,M$ such that the underlying 
\hbox{one\hh-}\hskip.7ptdi\-men\-sion\-al manifolds of the maximal integral 
curves of the restriction of\/ $\,v=\navp\,$ to\/ $\,\,U\smallsetminus\sd\,$ 
coincide with the length\/ $\,\ve\,$ open geodesic segments emanating from\/ 
$\,\sd\hh\cap U\,$ and normal to\/ $\,\sd$.
\item The gradient\/ $\,v=\navp\,$ is tangent to every nonconstant geodesic 
$\,[\hs0,b)\ni t\mapsto x(t)$ with\/ $\,x(0)=y\,$ and\/ 
$\,\dot x(0)\in N_y\sd$, where\/ $\,b\in(0,\infty\hh]$, and the set of\/ 
$\,t\in[\hs0,b)\,$ for which\/ $\,v_{x(t)}\nh=0\,$ is discrete.
\end{enumerate}
\end{lemma}
\begin{proof}Case (ii) in Lemma~\ref{diagz} for $\,v=\navp\,$ is excluded by 
self-ad\-joint\-ness of $\,B=[\nabla\nh v]_y$. Now (i) and (ii) are immediate 
from Lemma~\ref{diagz}(i) and the rank condition in the definition of a 
\hbox{Morse\hs-}\hskip0ptBott function. Note that $\,B\ne0$, for otherwise 
$\,\sd\,$ would be both a sub\-man\-i\-fold of co\-di\-men\-sion $\,0\,$ and a 
closed subset of $\,M$, which is not possible as $\,\sd\ne M$.

Assertion (iii) is a trivial consequence of Remark~\ref{invma}, since (ii) and 
\cite[Lemma~8.2]{derdzinski-maschler-06} imply that $\,\navp$ is tangent to 
all sufficiently short geodesic segments normal to $\,\sd$.

For $\,b\,$ and $\,x(t)\,$ as in (iv), let $\,t_{\mathrm{sup}}$ be the 
supremum of $\,t'\in(0,b)\,$ such that $\,v\,$ is tangent to the geodesic 
segment $\,[\hs0,t'\hh]\ni t\mapsto x(t)\,$ and the set of 
$\,t\in[\hs0,t'\hh]\,$ with $\,v_{x(t)}\nh=0\,$ is finite. By (iii), 
$\,t_{\mathrm{sup}}>0$.

Suppose now that $\,t_{\mathrm{sup}}<b$. The word `supremum' then can be 
replaced with `maximum' since, whether $\,v\ne0\,$ or $\,v=0\,$ at the point 
$\,x(t_{\mathrm{sup}})$, the parameter values 
$\,t\in[\hs0,t_{\mathrm{sup}})\,$ with $\,v_{x(t)}\nh=0\,$ cannot form a 
strictly increasing sequence that converges to $\,t_{\mathrm{sup}}$. (In the 
former case this follows from continuity of $\,v$, in the latter from (iii) 
applied to $\,y\hh'\nh=x(t_{\mathrm{sup}})$ and the critical manifold 
containing $\,y\hh'\nnh$, rather than $\,y\,$ and $\,\sd$.) Next, maximality 
of $\,t_{\mathrm{sup}}$ gives $\,v=0\,$ at $\,y\hh'\nnh$. Applying (iii), 
again, to $\,y\hh'$ instead of $\,y$, we see that $\,v\,$ is tangent to some 
segment $\,[\hs0,t'\hh]\ni t\mapsto x(t)\,$ with $\,t'>t_{\mathrm{sup}}$. The 
resulting contradiction shows that $\,t_{\mathrm{sup}}=b$, completing the 
proof.
\end{proof}
\begin{remark}\label{dtzez}For $\,(M,g),\vp,\sd,$ and $\,y\,$ satisfying the 
hypotheses of Lemma~\ref{mrbgg}, and any unit-speed geodesic 
$\,t\mapsto x(t)\,$ such that $\,x(0)=y\,$ and $\,\dot x(0)\in N_y\sd$, 
writing $\,\dot\vp(t)=d\hh[\hh\vp(x(t))]/dt$, we get, from 
Lemma~\ref{mrbgg}(ii),
\begin{equation}\label{dtz}
\dot\vp(0)\,=\,0\,\ne\,\ddot\vp(0)\,=\,a\hh,\hskip6pt\mathrm{with}\hskip6pta
\hskip6pt\mathrm{as\ in\ Lemma~\ref{mrbgg}(i).}
\end{equation}
\end{remark}

\section{An $\,\bbRP^1\nnh$-valued Invariant}\label{ri}
Any nonconstant Kil\-ling potential with a geodesic gradient on a K\"ah\-ler 
surface $\,(M,g)\,$ naturally gives rise to a $\,C^\infty$ mapping 
$\,\yj:M\to\bbRP^1\nnh$, described in Lemma~\ref{dvgez} below. We begin by 
introducing some notations.

In the remainder of the paper, except Section~\ref{cm}, $\,\vp\,$ is always 
assumed to be a nonconstant Kil\-ling potential with a geodesic gradient on a 
K\"ah\-ler manifold $\,(M,g)$ of complex dimension $\,m\ge2$. We write
\begin{equation}\label{not}
v=\navp\hs,\hskip27ptu=Jv\hs,\hskip27ptQ=g(v,v)\hs.
\end{equation}
The open set $\,M'\nh\subset M\,$ on which $\,v\ne0\,$ is connected and 
dense in $\,M$, cf.\ \cite[Remark~2.3(ii)]{derdzinski-maschler-06}. On 
$\,M'$ one has the distributions $\,\mathcal{V}=\mathrm{Span}\hs(v,u)\,$ 
and $\,\mathcal{H}=\mathcal{V}^\perp\nnh$. At any point of $\,M'\nnh$, nonzero 
vectors in $\,\mathcal{V}\,$ are eigenvectors of $\,\nabla\nh v\,$ for the 
eigenvalue function $\,\ta\,$ appearing in (\ref{nvv}). Furthermore,
\begin{equation}\label{tps}
\begin{array}{l}
\mathrm{a)}\hskip6pt2\ta\,=\,dQ\hh/\nh d\vp\hh,\hskip19pt\mathrm{b)}
\hskip6ptd_v\vp\,=\,Q\hh,\hskip19pt\mathrm{c)}\hskip6ptd_vQ\,=\,2\ta\hh Q\hh,\\
\mathrm{d)}\hskip6ptg(v,v)\,=\,g(u,u)\,=\,Q\hh,\hskip18ptg(v,u)\,=\,0\hh,
\end{array}
\end{equation}
where (\ref{tps}.a) makes sense in view of the line following (\ref{qeg}). 
In fact, (\ref{not}) yields (\ref{tps}.b) and (\ref{tps}.d), while 
(\ref{tnd}), (\ref{nvv}) and (\ref{not}) give $\,dQ=2\ta\,d\vp$, so that 
(\ref{tps}.a) and (\ref{tps}.c) follow.

If $\,m=2$, nonzero vectors in $\,\mathcal{H}\,$ are also eigenvectors of 
$\,\nabla\nh v$, for the eigenvalue function $\,\si\,$ given by 
$\,2\si=\Delta\vp-\hs2\ta$. Thus,
\begin{equation}\label{dps}
\mathrm{i)}\hskip6pt\Delta\vp\,=\,2\hh(\ta+\si)\hh,\hskip19pt\mathrm{ii)}
\hskip6pt|\nabla\nh v|^2\,=\,\hs2\hh(\ta^2\nh+\si^2)\hh.
\end{equation}
(The vec\-tor-bun\-dle morphism $\,\nabla\nh v:\tm\to\tm\,$ is 
com\-plex-lin\-e\-ar and Her\-mit\-i\-an at every point; see 
Section~\ref{kp}.) Since $\,\Delta\vp=\mathrm{div}\hskip2ptv$, 
(\ref{dvd}) combined with (\ref{nvv}) implies, whenever $\,m\ge2$, that 
$\,d_v\hs\Delta\vp=2\hs(d_v\ta+\ta\Delta\vp-|\nabla\nh v|^2\hh)$. 
Consequently, by (\ref{dps}),
\begin{equation}\label{dvp}
d_v\hs\si\,=\,2\hh(\ta-\si)\hh\si\hskip18pt\mathrm{if}\hskip10ptm=2\hh.
\end{equation}
\begin{lemma}\label{dvgez}For any nonconstant Kil\-ling potential\/ $\,\vp\,$ 
with a geodesic gradient on a K\"ah\-ler surface\/ $\,(M,g)$, there exists a 
unique\/ $\,C^\infty$ mapping\/ $\,\yj:M\to\bbRP^1$ such that, with the 
conventions of Remark\/~{\rm\ref{rpone}}, 
$\,\yj=\vp-Q\hs/\nh(\Delta\vp-\hs2\ta)\,$ on\/ $\,M'\nnh$. In addition,
\begin{enumerate}
  \def\theenumi{{\rm\alph{enumi}}}
\item At every point\/ $\,x\in M$, the vectors\/ $\,v_x\hs$ and\/ $\,u_x\hh$ 
lie in\/ $\,\mathrm{Ker}\hskip2.7ptd\yj_x$.
\item $\yj\,$ is constant along every geodesic issuing from a critical 
manifold\/ $\,\sd\,$ of $\,\vp\,$ in a direction normal to\/ $\,\sd$, cf.\ 
Example\/~{\rm\ref{kpamb}}.
\item $\yj\,$ is constant on\/ $\,M\,$ if and only if\/ $\,\vp\,$ is a 
special K\"ah\-\hbox{ler\hs-}\hskip0ptRic\-ci potential.
\end{enumerate}
\end{lemma}
\begin{proof}We begin by establishing (a) and (c) for $\,M'$ rather than 
$\,M$. Clearly, (a) holds if $\,x\,$ lies in the interior of the set on which 
$\,\yj=\infty$. On the set where $\,\yj\ne\infty$, treating 
$\,\yj=\vp-Q\hs/\nh(2\si)\,$ as a real-val\-ued function, we clearly have 
$\,d_u\yj=0\,$ since $\,u\,$ is a Kil\-ling field and $\,d_u\vp=0\,$ by 
(\ref{not}) combined with (\ref{tps}.d), while $\,d_v\yj=0\,$ due to 
(\ref{tps}.b), (\ref{tps}.c) and (\ref{dvp}). As the union of the two open 
sets is dense in $\,M'\nnh$, (a) on $\,M'$ follows.

To prove (c) on $\,M'\nnh$, assume first that $\yj:M'\nh\to\bbRP^1$ is 
constant. Both when $\,\yj=\infty\,$ (and so $\,\Delta\vp=\hs2\ta$), and 
when $\,\yj\ne\infty$, this implies that $\,\Delta\vp\,$ is, locally in 
$\,M'\nnh$, a function of $\,\vp$, since so are $\,Q\,$ and $\,\ta\,$ by 
(\ref{qeg}) and (\ref{tps}.a). In view of Remark~\ref{skrsu}, $\,\vp\,$ then 
is a special K\"ah\-\hbox{ler\hs-}\hskip0ptRic\-ci potential. On the other 
hand, if $\,\vp\,$ is a special K\"ah\-\hbox{ler\hs-}\hskip0ptRic\-ci 
potential, we either have $\,\si=0\,$ identically on $\,M'\nnh$, or 
$\,\si\ne0\,$ everywhere in $\,M'$ \cite[Lemma 12.5]{derdzinski-maschler-03}. 
As $\,\yj=\vp-Q\hs/\nh(2\si)$, in the former case $\,\yj=\infty$, and in the 
latter $\,\yj\,$ is a real constant \cite[Lemma 12.5]{derdzinski-maschler-03}, 
which yields (c).

We now show that $\yj:M'\nh\to\bbRP^1$ has a $\,C^\infty$ extension to $\,M$. 
To this end, let $\,\sd\,$ be the critical manifold of $\,\vp\,$ containing a 
given point $\,y\in M\smallsetminus M'\nnh$, cf.\ Example~\ref{kpamb}. For  
$\,\sd'\nnh,\ve\,$ and $\,\,U\,$ chosen as in Remark~\ref{invma}, 
$\,\,U\smallsetminus\sd'$ is a bundle over $\,\sd'$ with fibres which are 
\hbox{even\hh-}\hskip.7ptdi\-men\-sion\-al (Example~\ref{kpamb}), and hence 
connected, punctured balls. By (a) for $\,v\,$ along with 
Lemma~\ref{mrbgg}(iv), the $\,C^\infty$ mapping 
$\,\yj:U\smallsetminus\sd'\nh\to\bbRP^1$ is constant on each fibre, so that it 
has an obvious $\,C^\infty$ extension to $\,\,U$, as required.

Finally, Lemma~\ref{mrbgg}(iv) and (a) for $\,v\,$ imply (b).
\end{proof}
For $\,(M,g)\,$ and $\,\vp\,$ constructed in Section~\ref{fe}, $\,\yj\,$ used 
in the construction, when viewed as a mapping $\,M\to\bbRP^1\nnh$, coincides 
with $\,\yj\,$ defined in Lemma~\ref{dvgez}. This is clear from 
(\ref{nab}), (\ref{dte}) and (\ref{tps}.a).

We will show later (Lemma~\ref{cdone}) that, if $\,M\,$ in Lemma~\ref{dvgez} 
is compact, the values of $\,\yj\,$ lie in 
$\,\bbRP^1\smallsetminus\bbI^\circ\nnh$, where 
$\,\bbI^\circ\nh=(\vp_{\mathrm{min}},\vp_{\mathrm{max}})$. Identifying 
$\,\bbRP^1\smallsetminus\bbI^\circ$ with an interval in $\,\bbR$, we may then 
treat $\,\yj\,$ as real-val\-ued invariant. However, such an adjustment is not 
possible in general, since $\,\yj:M\to\bbRP^1$ is {\it surjective\/} for some 
nonconstant Kil\-ling potentials $\,\vp\,$ with geodesic gradients on 
(noncompact) K\"ah\-ler surfaces $\,(M,g)$. An example arises when one 
modifies the construction in Section~\ref{fe}, as described in the second 
paragraph of Remark~\ref{bspnt}. Specifically, let $\,\sd=\bbC$, and so 
$\,\sd=\,U_+\nh\cup U_-$, where the open set $\,\,U_\pm$ is defined by the 
condition $\,\pm\hs\mathrm{Re}\,z<1\,$ imposed on $\,z\in\bbC$. We choose 
$\,\yj:\bbC\to\bbRP^1$ to be a surjective mapping such that $\,\yj=\infty\,$ 
on the closure $\,K\,$ of $\,U_+\nh\cap U_-$, while $\,\yj\,$ restricted to 
$\,\bbC\smallsetminus K\,$ is real-val\-ued and has no critical points, and, 
finally, neither $\,\yj:U_+\nh\to\bbRP^1$ nor 
$\,\yj:U_-\nh\to\bbRP^1$ is surjective. (For instance, $\,\yj\,$ with the 
above properties may be a function of $\,\mathrm{Re}\,z$.) We now select base 
points $\,\vp_*^\pm\in\bbR\smallsetminus\yj(U_\pm)$, any metric $\,h\,$ on 
$\,\sd=\bbC$, and any $\,a\in(0,\infty)$. The $\,2$-form $\,\varOmega\,$ on 
$\,\sd\,$ equal to $\,-\hs a\hs(\vp_*^\pm-\yj)^{-1}\omega^{(h)}$ on 
$\,\,U_\pm$ is well defined, since both expressions yield $\,\varOmega=0\,$ on 
$\,U_+\nh\cap U_-$. Being closed, $\,\varOmega\,$ is exact, and so it the 
curvature form of a Hermitian connection in the trivial complex line bundle 
$\,\mathcal{L}\,$ over $\,\sd$, with the bundle projection still denoted by 
$\,\prj:\mathcal{L}\to\sd$. We now define a metric $\,g\,$ on an open subset 
$\,M^\pm$ of the line bundle $\,\mathcal{L}\nh^\pm\nh=\prj^{-1}(U_\pm)\,$ over 
$\,U_\pm$ as in Remark~\ref{bspnt}, using $\,\vp_*^\pm$ and the same function 
$\,Q\,$ of the variable $\,\vp\,$ in both cases. As the two metrics agree on 
the intersection $\,\prj^{-1}(U_+\nh\cap U_-)$, they together form a metric 
$\,g\,$ on $\,M=M^+\nh\cup M^-\nnh$, thus giving rise to a triple 
$\,(M,g,\vp)\,$ for which $\,\yj:M\to\bbRP^1$ is surjective.
\begin{remark}\label{ltref}For later reference, note that, under the 
hypotheses made in the lines preceding (\ref{not}), if $\,m=2$, the 
$\,\mathcal{V}\,$ component $\,[w,w\hh'\hh]^{\mathcal{V}}\nnh$, relative to 
the decomposition $\,\tm'\nh=\mathcal{H}\oplus\mathcal{V}$, of the Lie bracket 
of any two sections $\,w,w\hh'$ of $\,\mathcal{H}\,$ is given by
\begin{equation}\label{wwv}
Q\hs[w,w\hh'\hh]^{\mathcal{V}}\hs=\,\,-2\hs\si\hs g(Jw,w\hh'\hh)\hs u\hh.
\end{equation}
If, in addition, $\,w,w\hh'$ commute with both $\,v\,$ and $\,u$, then
\begin{equation}\label{dvq}
d_v[\hs\si\hs g(w,w\hh'\hh)/Q\hh]\,=\,d_u[\hs\si\hs g(w,w\hh'\hh)/Q\hh]\,
=\,0\hh.
\end{equation}
Both equalities follow since $\,\si\,$ is the eigenvalue function of 
$\,\nabla\nh v\,$ in $\,\mathcal{H}$, and so
\begin{equation}\label{dvg}
g(\nabla_{\!w}v,w\hh'\hh)=\si\hs g(w,w\hh'\hh)\hh,\hskip13pt
g(\nabla_{\!w}u,w\hh'\hh)=g(J\nabla_{\!w}v,w\hh'\hh)=\si\hs g(Jw,w\hh'\hh)
\end{equation} 
for sections $\,w,w\hh'$ of $\,\mathcal{H}$. Hence, as 
$\,g(v,\nabla_{\!w}w\hh'\hh)=-\hh g(\nabla_{\!w}v,w\hh'\hh)\,$ and 
$\,g(u,\nabla_{\!w}w\hh'\hh)=-\hh g(\nabla_{\!w}u,w\hh'\hh)$, we have 
$\,g(v,\nabla_{\!w}w\hh'\hh)=-\hs\si\hs g(w,w\hh'\hh)\,$ and 
$\,g(u,\nabla_{\!w}w\hh'\hh)=-\hs\si\hs g(Jw,w\hh'\hh)$. 
Skew-sym\-me\-triz\-ed in $\,w,w\hh'\nnh$, this gives (\ref{wwv}) due 
to symmetry of $\,g(w,w\hh'\hh)\,$ and skew-sym\-me\-try of 
$\,g(Jw,w\hh'\hh)\,$ in $\,w,w\hh'\nnh$. For the same reasons of 
(skew)-sym\-me\-try, assuming that $\,w,w\hh'$ commute with $\,v,u$, we obtain 
$\,d_v[\hh g(w,w\hh'\hh)]=2\hs\si\hs g(w,w\hh'\hh)\,$ and 
$\,d_u[\hh g(w,w\hh'\hh)]=0\,$ in view of (\ref{dvg}) and the Leib\-niz rule. 
Now (\ref{tps}.c) and (\ref{dvp}) yield (\ref{dvq}).
\end{remark}
\begin{remark}\label{evnbu}Let $\,\vp\,$ be a nonconstant Kil\-ling potential 
with a geodesic gradient on a K\"ah\-ler surface $\,(M,g)$. If $\,\sd\,$ is a 
critical manifold of $\,\vp$, cf.\ Example~\ref{kpamb}, and $\,y\in\sd$, then 
the covariant derivative $\,[\nabla\nh u]_y:T_yM\to T_yM\,$ of the Kil\-ling 
field $\,u=J(\navp)\,$ at $\,y\,$ has the kernel $\,T_y\sd$, and acts as the 
operator $\,aJ_y$ in the normal space $\,N_y\sd$, where $\,a\,$ is the unique 
nonzero eigen\-value of $\,\nabla d\vp\,$ at $\,y$, cf.\ 
Lemma~\ref{mrbgg}(i)\hh--\hh(ii).

In fact, $\,\nabla d\vp\,$ corresponds via $\,g\,$ to $\,\nabla\nh v$, for 
$\,v=\navp$, while $\,\nabla\nh u=J\circ\nabla\nh v$.
\end{remark}

\section{Morse\hs-Bott Functions on Compact Manifolds}\label{cm}
We now consider \hbox{Morse\hs-}\hskip0ptBott functions $\,\vp\,$ with 
geodesic gradients such that
\begin{equation}\label{cdo}
\mathrm{all\ critical\ manifolds\ of\ }\,\,\vp\hskip1.7pt\,\mathrm{\ are\ of\ 
co\-di\-men\-sions\ greater\ than\ }\,1\hh.
\end{equation}
In view of Example~\ref{kpamb}, given a function $\,\vp\,$ on a K\"ah\-ler 
manifold $\,(M,g)$,
\begin{equation}\label{cdk}
\mathrm{condition\ (\ref{cdo})\ holds\ whenever\ }\,\,\vp\,\,\mathrm{\ is\ 
a\ nonconstant\ Kil\-ling\ potential.}
\end{equation}
\begin{lemma}\label{smdhs}If the Hess\-i\-an of a 
\hbox{Morse\hs-}\hskip0ptBott function\/ $\,\vp\,$ on a compact manifold is 
sem\-i\-def\-i\-nite at every critical point, and all critical manifolds 
are of co\-di\-men\-sions\/ $\,k>1$, then
\begin{enumerate}
  \def\theenumi{{\rm\alph{enumi}}}
\item $\vp\,$ has exactly two critical manifolds, which are its maximum and 
minimum levels,
\item all levels of\/ $\,\vp\,$ are connected.
\end{enumerate}
\end{lemma}
\begin{proof}See \cite[Proposition 11.4]{derdzinski-maschler-06}.
\end{proof}
\begin{theorem}\label{qfcot}Suppose that\/ $\,\vp\,$ is a 
\hbox{Morse\hs-}\hskip0ptBott function with a geodesic gradient on a 
compact Riemannian manifold\/ $\,(M,g)\,$ and all critical manifolds of\/ 
$\,\vp\,$ have co\-di\-men\-sions greater than $\,1$. Let us also set\/ 
$\,\bbI=[\hh\vp_{\mathrm{min}},\vp_{\mathrm{max}}]\,$ and\/ 
$\,\bbI^\circ\nh=(\vp_{\mathrm{min}},\vp_{\mathrm{max}})$. Then
\begin{enumerate}
  \def\theenumi{{\rm\roman{enumi}}}
\item $Q=g(\navp,\navp)\,$ is a\/ $\,C^\infty$ function 
of\/ $\,\vp$, in the sense of Remark\/~{\rm\ref{cifot}},
\item for\/ $\,y,a\,$ as in Lemma\/~{\rm\ref{mrbgg}(i)}, and\/ 
$\,\vp\mapsto Q\,$ as in\/ {\rm(i)}, $\,dQ\hh/\nh d\vp\,$ at\/ $\,y\,$ 
equals\/ $\,2a$,
\item for the function\/ $\,\vp\mapsto Q\,$ in\/ {\rm(i)}, the integral\/ 
$\,\ds\,$ of\/ $\,Q^{-1/2}\hs$ over\/ $\,\bbI\,$ is finite,
\item $\ds\,$ in\/ {\rm(iii)} is the distance between the minimum and maximum 
levels of\/ $\,\vp$,
\item the assignment\/ $\,\vp\mapsto s$, characterized by\/ 
$\,ds\hh/\nh d\vp=Q^{-1/2}\hs$ and\/ $\,s=0\,$ at\/ 
$\,\vp=\vp_{\mathrm{min}}\hs$, is a homeo\-mor\-phism\/ 
$\,\bbI\to[\hs0,\ds]\,$ which maps\/ $\,\bbI^\circ$ 
dif\-feo\-mor\-phic\-al\-ly onto\/ $\,(0,\ds)$,
\item $s\,$ in\/ {\rm(v)} equals the distance from the minimum level of\/ 
$\,\vp$, when treated, due to its dependence on $\,\vp$, as a function\/ 
$\,s:M\to\bbR$.
\end{enumerate}
\end{theorem}
\begin{proof}Let $\,\sd\,$ and $\,\sd^*$ be the minimum and maximum levels 
of $\,\vp$.

By (\ref{qeg}), $\,Q\,$ restricted to the open set $\,M'$ where 
$\,d\vp\ne0\,$ is, locally, a $\,C^\infty$ function of $\,\vp$. The word 
`locally' can be dropped in view of Lemma~\ref{smdhs}(b). The resulting 
$\,C^\infty$ function $\,\bbI^\circ\nh\ni\vp\mapsto Q\,$ has a continuous 
extension to $\,\bbI$, equal to $\,0\,$ at the endpoints.

Next, let us fix a parametrization $\,[\hs0,\delta\hh]\ni t\mapsto x(t)\,$ 
of a shortest geodesic segment $\,\varGamma\,$ joining $\,\sd\,$ to 
$\,\sd^*\nnh$, with $\,x(0)\in\sd$. By (\ref{dtz}), the infimum $\,t'$ of 
those $\,t\in(0,\delta)\,$ for which $\,\dot\vp(t)=0\,$ lies 
in $\,(0,\delta\hh]$. As $\,v=\navp\,$ is tangent to $\,\varGamma\,$ 
(Lemma~\ref{mrbgg}(iv)), and $\,\dot\vp=g(v,\dot x)\,$ vanishes at 
$\,t=t'\nnh$, at $\,x(t'\hh)\,$ we must also have $\,v=0$, and hence 
$\,\vp=\vp_{\mathrm{max}}$. (The fact that $\,\vp(x(t))\,$ is an increasing 
function of $\,t\in(0,t'\hh)\,$ excludes the only other possibility left 
open by Lemma~\ref{smdhs}(a), namely, $\,\vp=\vp_{\mathrm{min}}$.) The 
dis\-tance-min\-i\-miz\-ing property of $\,\varGamma\hs$ now implies 
that $\,t'\nh=\delta$, and so $\,v_{x(t)}\nh\ne0\,$ whenever 
$\,t\in(0,\delta)$, that is, the open-in\-ter\-val restriction 
$\,(0,\delta)\ni t\mapsto x(t)\,$ is a re\-pa\-ram\-e\-trized integral curve 
of the gradient $\,v=\navp$. Thus, $\,\ds\,$ in (iii) is finite, as it equals 
the length of $\,\varGamma\,$ (see Remark~\ref{lngth}), which proves (iii) and 
(iv). Assertion (v) is in turn obvious from (iii). Finally, let us fix 
$\,x\in M'\nnh$. According to Remark~\ref{lngth} and (iii), the length of 
the maximal integral curve of $\,v\,$ through $\,x\,$ is finite, and so its 
underlying \hbox{one\hh-}\hskip.7ptdi\-men\-sion\-al manifold $\,C\,$ has 
limit endpoints $\,y_{\mathrm{min}}$ and $\,y_{\mathrm{max}}$ 
(Remark~\ref{compl}), at which $\,\vp=\vp_{\mathrm{min}}$ and 
$\,\vp=\vp_{\mathrm{max}}$ due to maximality of $\,C\,$ and 
Lemma~\ref{smdhs}(a). By Remark~\ref{lngth}, the length of $\,C\,$ is 
$\,\ds$. Hence, in view of (iv), 
$\,\varGamma\nh=C\cup\{y_{\mathrm{min}},y_{\mathrm{max}}\}\,$ is a 
dis\-tance-min\-i\-miz\-ing geodesic segment. Consequently, the same is true 
of the sub\-seg\-ment $\,\varGamma\hh'$ of $\,\varGamma\,$ joining 
$\,y_{\mathrm{min}}$ to $\,x$, which is also the shortest geodesic segment 
joining $\,\sd\,$ to $\,x$. The distance between $\,\sd\,$ and $\,x\,$ is 
therefore given by the length formula in Remark~\ref{lngth}, applied to 
$\,\varGamma\hh'\nnh$, and (vi) follows.

For $\,(\nh-\ve,\ve)\ni t\mapsto x(t)\,$ as in Remark~\ref{dtzez}, with 
$\,\ve\in(0,\infty)\,$ chosen sufficiently small, $\,|t|\,$ equals 
$\,\mathrm{dist}\hh(\sd,x(t))\,$ (or, $\,\mathrm{dist}\hh(\sd^*\nnh,x(t))$), 
cf.\ Remark~\ref{gauss} and Lemma~\ref{smdhs}(a). Thus, by (vi), $\,|t|\,$ is 
the value of $\,s:M\to\bbR\,$ or, respectively, $\,\ds-s:M\to\bbR$, at 
$\,x(t)$. (Note that replacing $\,\vp\,$ by $\,\vp_*\nh-\vp$, where $\,\vp_*$ 
is the midpoint of $\,\bbI$, causes $\,\vp_{\mathrm{min}}$ to be switched with 
$\,\vp_{\mathrm{max}}$, and $\,s\,$ with $\,\ds-s$.) The homeo\-mor\-phic 
correspondence between $\,s\,$ and $\,\vp\,$ in (v) now implies that 
$\,\vp(x(t))\,$ is an even $\,C^\infty$ function of $\,t$, and, due to the 
al\-read\-y\nh-es\-tab\-lish\-ed dependence of $\,Q\,$ on $\,\vp$, the same 
is true of $\,Q(x(t))$. Evenness of both functions and the relation 
$\,\dot\vp(0)=0\ne\ddot\vp(0)\,$ (cf.\ (\ref{dtz})) are well-known to imply 
that $\,Q\,$ restricted to some neighborhood of $\,\vp_{\mathrm{min}}$ (or, 
$\,\vp_{\mathrm{max}}$) in $\,\bbI\,$ is a $\,C^\infty$ function of $\,\vp$. 
See, for instance, \cite[the last nine lines in \S9]{derdzinski-maschler-06}. 
Thus, the extension of $\,Q\,$ from $\,\bbI^\circ$ to $\,\bbI\,$ is of class 
$\,C^\infty\nnh$, which proves (i).

Finally, $\,dQ\hh/\nh d\vp=2\ta\,$ on $\,\bbI^\circ$, and, consequently, on 
$\,\bbI$, since $\,dQ=2\ta\,d\vp\,$ by (\ref{tnd}) and (\ref{nvv}). Again, let 
us choose a geodesic $\,t\mapsto x(t)\,$ as in Remark~\ref{dtzez}. Then 
$\,v\,$ is tangent to it (Lemma~\ref{mrbgg}(iv)) and so, by (\ref{nvv}), 
$\,\dot x\,$ is, at every $\,t$, an eigenvector of $\,\nabla d\vp\,$ (that is, 
of $\,\nabla\nh v$) for the eigenvalue 
$\,\ta=[\nabla d\vp](\dot x,\dot x)=\ddot\vp$. Now (\ref{dtz}) implies (ii).
\end{proof}
The next lemma uses the notations of Remark~\ref{gauss} and $\,\ds\,$ 
defined in Theorem~\ref{qfcot}.
\begin{lemma}\label{mrsbt}Let\/ $\,\sd\,$ and\/ $\,\sd^*$ be the minimum and 
maximum levels of a nonconstant \hbox{Morse\hs-}\hskip0ptBott function\/ 
$\,\vp\,$ with a geodesic gradient and\/ {\rm(\ref{cdo})} on a compact 
Riemannian manifold $\,(M,g)$. Then\/ $\,\mathrm{Exp}^\perp$ maps\/ 
$\,N^\ds\nh\sd\,$ dif\-feo\-mor\-phic\-al\-ly onto\/ $\,B_\ds\nh(\nh\sd)$, 
and\/ $\,B_\ds\nh(\nh\sd)=M\smallsetminus\sd^*\nnh$.
\end{lemma}
\begin{proof}That $\,B_\ds\nh(\nh\sd)=M\smallsetminus\sd^*$ is obvious from 
assertions (v) and (vi) in Theorem~\ref{qfcot}.

Let $\,M'\nh\subset M\,$ be the open set given by $\,v\ne0$, where 
$\,v=\navp$. If $\,x\in M'\nnh$, the geodesic segment 
$\,[\hs0,1\hh]\ni t\mapsto x(t)\,$ of length $\,\mathrm{dist}\hh(\sd,x)$, such 
that $\,x(0)=x\,$ and $\,\dot x(0)$ is a negative multiple of $\,v_x$, is also 
a shortest segment connecting $\,x\,$ to $\,\sd$. In fact, choosing a shortest 
segment $\,\varGamma\,$ connecting $\,x\,$ to $\,\sd$, we see that it is 
normal to $\,\sd$, and so $\,v\,$ is tangent to it (Lemma~\ref{mrbgg}(iv)); as 
the dif\-feo\-mor\-phism $\,\bbI^\circ\nh\to(0,\ds)\,$ in 
Theorem~\ref{qfcot}(v) is strictly increasing, on 
$\,\varGamma\smallsetminus\sd\,$ the gradient $\,v=\navp\,$ must, by 
Theorem~\ref{qfcot}(vi), point away from $\,\sd\,$ and toward $\,x$. Thus, 
both geodesic segments satisfy the same initial conditions at $\,x$.

Let the mapping $\,H:M'\nh\to\tm\,$ send any $\,x\in M'$ to the vector 
$\,-\hh\dot x(1)\,$ tangent to $\,M\,$ at $\,x(1)$, for $\,t\mapsto x(t)\,$ 
associated with $\,x\,$ as in the last paragraph. Since $\,x(1)\in\sd$ and 
$\,\dot x(1)\,$ is normal to $\,\sd\,$ (see above), $\,H\,$ takes values in 
the subset $\,N^\ds\nh\sd\smallsetminus\sd\,$ of $\,\tm$. Our claim now 
follows, since $\,H\circ\mathrm{Exp}^\perp$ and 
$\,\mathrm{Exp}^\perp\nnh\circ H\,$ are easily seen to be the identity 
mappings of $\,N^\ds\nh\sd\smallsetminus\sd\,$ and 
$\,M'\nh=B_\ds\nh(\nh\sd)\smallsetminus\sd$, while, if 
$\,\ve\in(0,\infty)\,$ is sufficiently small, 
$\,\mathrm{Exp}^\perp:N^{\hh\ve}\nh\sd\to B_\ve\nh(\nh\sd)\,$ is a 
dif\-feo\-mor\-phism (Remark~\ref{gauss}).
\end{proof}

\section{Proof of Theorem~\ref{clssf}, first part}\label{pf}
In this section we construct the required data (\ref{dat}) for any triple 
$\,(M,g,\vp)\,$ satisfying the assumptions of Theorem~\ref{clssf}, and verify 
conditions (i) -- (vi) in Section~\ref{fe}.
\begin{lemma}\label{circl}Let a nonconstant Kil\-ling potential\/ $\,\vp\,$ on 
a complete K\"ahler manifold\/ $\,(M,g)$ have a geodesic gradient. Then
\begin{enumerate}
  \def\theenumi{{\rm\roman{enumi}}}
\item at every critical point of\/ $\,\vp$, the Hess\-i\-an\/ 
$\,\nabla d\vp\,$ has exactly one nonzero eigen\-value, the absolute value of 
which is the same for all critical points,
\item if the set of critical points of\/ $\,\vp\,$ is nonempty, the flow of 
the Kil\-ling vector field\/ $\,u=J(\navp)\,$ is periodic.
\end{enumerate}
\end{lemma}
\begin{proof}Obvious from Lemma~\ref{mrbgg}(i) (cf.\ Example~\ref{kpamb}) and 
\cite[Corollary 10.3]{derdzinski-maschler-06}.
\end{proof}
\begin{lemma}\label{crlvl}If\/ $\,\vp\,$ is a nonconstant Kil\-ling potential 
with a geodesic gradient on a compact K\"ahler manifold\/ $\,(M,g)$, then, for 
some\/ $\,a\in(0,\infty)$,
\begin{enumerate}
  \def\theenumi{{\rm\alph{enumi}}}
\item $\vp_{\mathrm{max}}\hs$ and\/ $\,\vp_{\mathrm{min}}\hs$ are the only 
critical values of\/ $\,\vp$,
\item the\/ $\,\vp$-pre\-im\-ages of\/ $\,\vp_{\mathrm{max}}\hs$ and\/ 
$\,\vp_{\mathrm{min}}\hs$ are compact complex submanifolds of\/ $\,M\nh$,
\item $Q=g(\navp,\navp)\,$ is a $\,C^\infty$ function of\/ $\,\vp$, as 
defined in Remark\/~{\rm\ref{cifot}},
\item the values\/ of\/ $\,dQ\hh/\nh d\vp\,$ at\/ 
$\,\vp=\vp_{\mathrm{min}}\hs$ and\/ $\,\vp=\vp_{\mathrm{max}}\hs$ are $\,2a\,$ 
and $\,-2a$.
\end{enumerate}
\end{lemma}
\begin{proof}Assertions (a) and (b) are immediate consequences of 
Lemma~\ref{smdhs} combined with Example~\ref{kpamb} and (\ref{cdk}); (c) and 
(d) similarly follow from Theorem~\ref{qfcot}(i)\hh--\hh(ii) and the 
ab\-so\-\hbox{lute\hh-}\hskip0ptval\-ue clause in Lemma~\ref{circl}(i).
\end{proof}
\begin{lemma}\label{cdone}Given a nonconstant Kil\-ling potential\/ $\,\vp\,$ 
with a geodesic gradient on a compact K\"ahler surface\/ $\,(M,g)$, let us 
set\/ $\,\bbI=[\hh\vp_{\mathrm{min}},\vp_{\mathrm{max}}]\,$ and\/ 
$\,\bbI^\circ\nh=(\vp_{\mathrm{min}},\vp_{\mathrm{max}})$.
\begin{enumerate}
  \def\theenumi{{\rm\roman{enumi}}}
\item All values of\/ $\,\yj:M\to\bbRP^1\nnh$, defined in 
Lemma\/~{\rm\ref{dvgez}}, lie in\/ $\,\bbRP^1\smallsetminus\bbI^\circ\nnh$.
\item If\/ $\,\vp\,$ is not a special K\"ah\-\hbox{ler\hs-}\hskip0ptRic\-ci 
potential, then
\begin{enumerate}
  \def\theenumi{{\rm\alph{enumi}}}
\item the maximum and minimum levels of\/ $\,\vp\,$ both have complex 
dimension\/ $\,1$,
\item the values of\/ $\,\yj\,$ all lie in\/ $\,\bbRP^1\smallsetminus\bbI$.
\end{enumerate}
\end{enumerate}
\end{lemma}
\begin{proof}First, let $\,\yj(y)\in\bbI^\circ$ at some $\,y\in M$. By 
Theorem~\ref{qfcot}(v)\hh--\hh(vi), which can be used here in view of 
Example~\ref{kpamb} and (\ref{cdk}), $\,\mathrm{dist}\hh(\sd,y)\le\ds$, for 
the minimum level $\,\sd\,$ of $\,\vp$. Hence $\,y\,$ lies on a geodesic 
segment $\,\varGamma\hs$ of length $\,\ds\,$ emanating from $\,\sd\,$ and 
normal to $\,\sd$. Due to injectivity of $\,\mathrm{Exp}^\perp$ on 
$\,N^\ds\nh\sd\,$ (Lemma~\ref{mrsbt}), $\,\varGamma\hs$ also provides a 
shortest connection between $\,\sd\,$ and any point of $\,\varGamma\nh$. 
Therefore, the function $\,s\,$ of Theorem~\ref{qfcot}(vi), restricted to 
$\,\varGamma\nh$, serves as an arc-length parameter for $\,\varGamma\nh$. 
Theorem~\ref{qfcot}(v) (or, Lemma~\ref{dvgez}(b)) implies now that the 
$\,\vp$-im\-age of $\,\varGamma\,$ is $\,\bbI\,$ (or, respectively, that 
$\,\yj$ is constant on $\,\bbI$). Thus, $\,\varGamma\hs$ contains a point 
$\,x\,$ at which $\,\yj(x)=\vp(x)\in\bbI^\circ$ and, consequently, 
$\,Q(x)>0\,$ (cf.\ Lemma~\ref{crlvl}(a) and (\ref{not})). The equality 
$\,\yj(x)=\vp(x)$ contradicts in turn the definition of $\,\yj$, proving 
(i).

Next, if some critical manifold of $\,\vp\,$ (cf.\ Example~\ref{kpamb}) 
consisted of a single point, the Hopf-Ri\-now theorem and Lemma~\ref{dvgez}(b) 
would imply that $\,\yj\,$ is constant on $\,M$, thus making $\,\vp$ a special 
K\"ah\-\hbox{ler\hs-}\hskip0ptRic\-ci potential (Lemma~\ref{dvgez}(c)). This 
implies (ii\hh--\hh a).

Finally, if $\,\yj(y)=\vp_{\mathrm{min}}$ or $\,\yj(y)=\vp_{\mathrm{max}}$ 
at some $\,y\in M$, we may assume that $\,y\,$ is a critical point of 
$\,\vp\,$ and $\,\yj(y)=\vp(y)$, which is achieved by choosing $\,\varGamma\,$ 
as above and replacing $\,y\,$ with an endpoint of $\,\varGamma\nnh$. In view 
of (i), $\,\vp\ne\yj\,$ everywhere in the open set $\,M'\nh\subset M\,$ on 
which $\,d\vp\ne0$. A fixed geodesic $\,t\mapsto x(t)\,$ having the properties 
listed in Remark~\ref{dtzez}, for our $\,y$, and the equality 
$\,2\si=Q/(\vp-\yj)\,$ on $\,M'$ (immediate from the definition of $\,\yj\,$ 
in Lemma~\ref{dvgez}) now allow us to evaluate $\,2\si(y)\,$ via l'Hospital's 
rule, with $\,Q\,$ and $\,\vp-\yj\,$ both vanishing at $\,y=x(0)\,$ due to 
(\ref{not}). Consequently, $\,2\si(y)$ is the limit, as $\,t\to0$, of 
$\,\dot Q/(\dot\vp-\dot\yj)=(d_vQ)/(d_v\vp-d_v\yj)$, where we have used the 
`\nh dot\hh' notation of Remark~\ref{dtzez} and the fact that, since 
$\,v=\navp\,$ is tangent to the geodesic (Lemma~\ref{mrbgg}(iv)) and nonzero 
at $\,x(t)\,$ for $\,t\ne0\,$ close to $\,0\,$ (Remark~\ref{invma}), 
$\,d/dt\,$ equals a specific function of the variable $\,t\ne0\,$ times 
$\,d_v$. From (\ref{tps}.b), (\ref{tps}.c) and Lemma~\ref{dvgez}(a) we now 
obtain $\,2\si(y)=2\ta(y)$. The two eigenvalues of the Hess\-i\-an 
$\,\nabla d\vp\,$ at $\,y\,$ thus coincide, and so, according to 
Lemma~\ref{mrbgg}(i)\hh--\hh(ii), $\,T_yM\,$ is the normal space at $\,y\,$ of 
the critical manifold $\,\sd\,$ of $\,\vp\,$ containing $\,y$. Hence 
$\,\sd=\{y\}$ and, by (a), $\,\vp\,$ is a special 
K\"ah\-\hbox{ler\hs-}\hskip0ptRic\-ci potential, which yields (ii\hh--\hh b).
\end{proof}
For $\,(M,g,\vp)\,$ as in Theorem~\ref{clssf}, we now define the data 
(\ref{dat}) by choosing: $\,a\,$ and $\,\bbI\ni\vp\mapsto Q$, where 
$\,\bbI=[\hh\vp_{\mathrm{min}},\vp_{\mathrm{max}}]$, as in 
Lemma~\ref{crlvl}(c)\hh--\hh(d); $\,\sd\,$ to be the minimum level of $\,\vp$, 
with $\,\yj:\sd\to\bbRP^1$ obtained by restricting to $\,\sd\,$ the mapping 
$\,\yj$ introduced in Lemma~\ref{dvgez}, and with the metric $\,h\,$ on 
$\,\sd\,$ given by
\begin{equation}\label{het}
h\,=\,(\vp_{\mathrm{min}}\nh-\yj)^{-1}(\vp_*-\yj)\hs g\hh,
\end{equation}
$\vp_*\nh\in\bbI\,$ being the midpoint; the normal bundle $\,\mathcal{L}\,$ of 
$\,\sd\,$ with the Hermitian fibre metric $\,(\hskip2.2pt,\hskip1pt)$, the 
real part of which is $\,g\,$ (that is, $\,g\,$ restricted to 
$\,\mathcal{L}$); and, finally, the horizontal distribution $\,\mathcal{H}\,$ 
of the normal connection in $\,\mathcal{L}$. Lemmas~\ref{crlvl}(c)\hh--\hh(d) 
and~\ref{cdone}(ii) state that these objects satisfy conditions (i) -- (vii) 
in Section~\ref{fe} except for the equality 
$\,\varOmega\hs=-\hs a\hs(\vp_*-\yj)^{-1}\omega^{(h)}\nh$, which will be 
established in the next section.

\section{Proof of Theorem~\ref{clssf}, second part}\label{sp}
Using the data (\ref{dat}) just constructed for the given triple 
$\,(M,g,\vp)$, we also choose, as in Section~\ref{fe}, a $\,C^\infty$ 
dif\-feo\-mor\-phism 
$\,(\vp_{\mathrm{min}},\vp_{\mathrm{max}})\ni\vp\mapsto r\in(0,\infty)\,$ with 
$\,dr/d\vp=ar/Q$. Its inverse now gives rise to the composite 
$\,r\mapsto\vp\mapsto s$, for $\,\vp\mapsto s\,$ as in as in 
Theorem~\ref{qfcot}(v), allowing us to treat $\,s\,$ as a function of $\,r\,$ 
and write $\,s=\sigma(r)$, so that $\,r\mapsto \sigma(r)\,$ is a 
dif\-feo\-mor\-phism $\,(0,\infty)\to(0,\ds)$. This in turn leads to a 
fibre-pre\-serv\-ing dif\-feo\-mor\-phism 
$\,\theta:N\sd\smallsetminus\sd\to N^\ds\nh\sd\smallsetminus\sd\,$ of 
punc\-tured-disk bundles, which sends a vector $\,w\ne0\,$ normal to $\,\sd\,$ 
at any point to $\,\sigma(r)w/r$, where $\,r=|w|\,$ is the $\,g$-norm of 
$\,w$. For later reference, note that, according to Theorem~\ref{qfcot}(v),
\begin{equation}\label{dsr}
d\hh[\hs\sigma(r)]\hh/\nh dr\,=\,(ar)^{-1}Q^{1/2}\hskip12pt\mathrm{and}
\hskip8pt\sigma(0)\hs=\hs0\hh,\hskip12pt\mathrm{while}\hskip8pts\hs
=\hs\sigma(r)\hh.
\end{equation}
By Lemma~\ref{mrsbt}, Example~\ref{kpamb} and (\ref{cdk}), 
$\,F=\mathrm{Exp}^\perp\nnh\circ\hh\theta\,$ maps $\,N\sd\smallsetminus\sd\,$ 
dif\-feo\-mor\-phic\-al\-ly onto the open sub\-man\-i\-fold 
$\,M'\nh\subset M\,$ on which $\,d\vp\ne0$.

We now show that $\,F\,$ is a bi\-hol\-o\-mor\-phic isometry of 
$\,N\sd\smallsetminus\sd\subset N\sd=\mathcal{L}$, with the complex structure 
and metric obtained as in Section~\ref{fe} from the data (\ref{dat}), onto 
our $\,(M'\nnh,g)$, and that it sends the Kil\-ling potential with a 
geodesic gradient, described in Section~\ref{fe}, onto our $\,\vp$. The proof, split into three lemmas, closely follows the argument in 
\cite[\S\S15--16]{derdzinski-maschler-06}.

To minimize confusion, the hatted symbols 
$\,\hat M,\hat M'\nnh,\hat{\mathcal{V}},\hat{\mathcal{H}},\hat g,\hat J,
\hat v,\hat u\,$ stand for for the objects constructed in Section~\ref{fe} 
from our data (and from $\,\vp\mapsto r\,$ chosen above), which in 
Section~\ref{fe} appeared as $\,M,M'\nnh,\mathcal{V},\mathcal{H},g,J,v,u$. For 
$\,M,M'\nnh,\mathcal{V},\mathcal{H},g,v,J,u$, the meaning is now the same as 
in Section~\ref{ri}: they are associated with $\,(M,g)\,$ and the function 
$\,\vp:M\to\bbR$. However, $\,\vp,r\,$ and $\,s$, in their original form, are 
used not only for the independent variables ranging over 
$\,\bbI^\circ\nnh,(0,\infty)\,$ and $\,(0,\ds)$, but, along with $\,Q\,$ and 
$\,\yj$, also denote mappings defined on both manifolds $\,M'$ and 
$\,\hat M'\nnh$. Similarly, $\,\sd\,$ is treated as a sub\-man\-i\-fold both 
of $\,M\,$ (the minimum level of $\,\vp$) and of $\,\mathcal{L}=N\sd\,$ (the 
zero section). Again, $\,\prj:\mathcal{L}\to\sd\,$ is the bundle projection.
\begin{lemma}\label{fstqv}The dif\-feo\-mor\-phism\/ 
$\,F:\hat M'\nh\to M'$ sends the functions $\,s,\vp,Q\,$ and the mapping\/ 
$\,\yj\,$ defined on\/ $\,\hat M'$ to their analogs on\/ $\,M'\nnh$, and the 
vector field\/ $\,\hat v\,$ to\/ $\,v$. 
\end{lemma}
\begin{proof}In the case of $\,\yj\,$ this is clear from 
Lemma~\ref{dvgez}(b), since $\,F\,$ restricted to $\,\sd\,$ is the identity 
mapping.

Because of how we defined $\,\hat g\,$ on $\,\hat{\mathcal{V}}\,$ in 
Section~\ref{fe}, given $\,y\in\sd$, (\ref{dsr}) implies that a line segment 
of $\,g_y$-length $\,r\,$ emanating from $\,0\,$ in the normal space 
$\,N_y\sd\,$ has the $\,\hat g$-length $\,\sigma(r)$, which is at the same 
time the $\,g_y$-length of the segment's image under $\,\theta$. That image is 
also a segment in $\,N_y\sd\,$ issuing from $\,0$, and so 
$\,\mathrm{Exp}^\perp$ sends it to a geodesic segment of $\,g$-length 
$\,\sigma(r)\,$ in $\,(M,g)$, normal to $\,\sd\,$ at $\,y$. Since 
Theorem~\ref{qfcot}(vi) applies to both $\,(M,g,\vp)\,$ and 
$\,(\hat M,\hat g,\vp)$, our claim about $\,s\,$ follows from the 
dis\-tance-min\-i\-miz\-ing clause of Remark~\ref{gauss}.

As the homeo\-mor\-phic correspondence $\,\bbI\to[\hs0,\ds]\,$ of 
Theorem~\ref{qfcot}(v) holds in both $\,(M,g,\vp)\,$ and 
$\,(\hat M,\hat g,\vp)$, the 
same now follows for $\,\vp\,$ and $\,Q$. Finally, we just saw that $\,F\,$ 
sends line segments emanating from $\,0\,$ in the normal spaces of $\,\sd\,$ 
to normal $\,g$-ge\-o\-des\-ics issuing from $\,\sd$. Since $\,\hat v\,$ is 
tangent to the former (by definition), and $\,v=\navp$ to the latter (cf.\ 
Example~\ref{kpamb} and Lemma~\ref{mrbgg}(iv)), the $\,F\nh$-im\-age of 
$\,\hat v\,$ is the product of a function and $\,v$. That the function in 
question equals $\,1\,$ is in turn obvious from the normalizing condition 
(\ref{tps}.b), valid in both $\,(M,g,\vp)\,$ and $\,(\hat M,\hat g,\vp)$, 
along with our assertion, already established for $\,\vp\,$ and $\,Q$.
\end{proof}
\begin{lemma}\label{fsutu}The\/ $\,F$-im\-ages of\/ $\,\hat u\,$ and\/ 
$\,\hat{\mathcal{V}}\,$ are, respectively, $\,u\,$ and\/ $\,\mathcal{V}$, 
while\/ $\,\hat g\,$ and $\,\hat J$ restricted to\/ $\,\hat{\mathcal{V}}\,$ 
correspond under\/ $\,F\,$ to\/ $\,g\,$ and $\,J\,$ on\/ $\,\mathcal{V}$.
\end{lemma}
\begin{proof}Obviously, $\,\theta\,$ preserves $\,\hat u$, that is, the 
$\,\theta$-im\-age of $\,\hat u\,$ is the restriction of $\,\hat u\,$ to 
$\,N^\ds\nh\sd\smallsetminus\sd$. As $\,u\,$ is a Kil\-ling field, 
Remarks~\ref{kilxp} and~\ref{evnbu} combined with the definition of 
$\,\hat u\,$ (cf.\ Section~\ref{fe}) imply in turn that $\,\mathrm{Exp}^\perp$ 
sends $\,\hat u\,$ to $\,u$. Hence so does $\,F\nh$.

The rest of our assertion is now obvious from Lemma~\ref{fstqv}, since in both 
$\,(M,g,\vp)$ and $\,(\hat M,\hat g,\vp)\,$ we have the relations 
(\ref{tps}.d) and $\,\mathcal{V}=\mathrm{Span}\hs(v,u)\,$ or, respectively, 
their hatted versions.
\end{proof}
\begin{lemma}\label{fshgj}The assertion of Lemma\/~{\rm\ref{fsutu}} remains 
true also when\/ $\,\hat{\mathcal{V}}\,$ and\/ $\,\mathcal{V}\,$ are replaced 
by\/ $\,\hat{\mathcal{H}}\,$ and\/ $\,\mathcal{H}$, while the data\/ 
{\rm(\ref{dat})} constructed in Section\/~{\rm\ref{pf}} satisfy condition\/ 
{\rm(vii)} of Section\/~{\rm\ref{fe}}.
\end{lemma}
\begin{proof}Let us fix a $\,g$-unit vector field 
$\,t\mapsto w(t)\in N_{y(t)}\sd$, normal to $\,\sd$, defined along a curve 
$\,t\mapsto y(t)\in\sd$, and parallel relative to the normal connection in 
$\,\mathcal{L}=N\sd$. Since $\,\sd\,$ is totally geodesic in $\,(M,g)\,$ 
(see Example~\ref{kpamb}), the last condition reads $\,\nabla_{\!\dot y}w=0$, 
where $\,\nabla\,$ is the Le\-vi-Ci\-vi\-ta connection of $\,g$. The variable 
$\,t\,$ ranges over some given open interval $\,(b,c)$. For any 
$\,t\in(b,c)\,$ and $\,s\in(0,\ds)$, we define $\,x(t,s)\in M\,$ to be the 
$\,F$-im\-age of $\,rw(t)\,$ treated as an element of $\,\hat M'\nnh$, for the 
unique $\,r\in(0,\infty)\,$ with $\,s=\sigma(r)$. Thus, by the definition of 
$\,F$, we obtain a mapping
\begin{equation}\label{map}
(b,c)\times(0,\ds)\ni(t,s)\,\mapsto\,x(t,s)\,
=\,\mathrm{exp}_{y(t)}\hs sw(t)\in M\hh.
\end{equation}
We will use subscripts for its partial derivatives $\,x_t,x_s$, and their 
partial covariant derivatives $\,x_{ts},x_{ss}$, etc. All such derivatives are 
sections of the pull\-back of $\,\tm\,$ under the mapping (\ref{map}). The 
sub\-script-style partial (or, partial covariant) derivatives also make sense 
for functions (or, respectively, vector fields) on $\,M$, which amounts to 
differentiating the latter objects along each of the curves given by 
(\ref{map}) with fixed $\,s\,$ or fixed $\,t$. More details can be found in 
\cite[\S14]{derdzinski-maschler-06}.

Writing $\,\langle\,,\rangle\,$ instead of $\,g$, and denoting by 
$\,|\hskip3pt|\,$ the $\,g$-norm, we now have
\begin{enumerate}
  \def\theenumi{{\rm\alph{enumi}}}
\item $x_s=\hh Q^{-1/2}v$, \ \ $\,|v|=|u|=Q^{\hs1/2}\nnh$,
\item $\langle u,x_t\rangle\nh_s=2\hs\langle u,x_{st}\rangle$,
\item $\langle u,x_t\rangle\nh_s
=\hs2\langle u,x_t\rangle\hs\ta\hs Q^{-1/2}\nnh$,
\item $Q_s=\hs2\ta\hs Q^{\hs1/2}$.
\end{enumerate}
Although equalities (a) -- (d) all appear in 
\cite[p.\ 101]{derdzinski-maschler-06}, they have to be established here 
independently, as \cite{derdzinski-maschler-06} makes a stronger assumption 
about $\,\vp$. However, the argument is the same as in 
\cite{derdzinski-maschler-06}.

First, (\ref{tps}.d) implies the second part of (a), and the first part then 
follows: by (\ref{map}) and Lemma~\ref{fstqv}, $\,v\,$ equals a positive 
function times $\,x_s$, and $\,|x_s|=1$. Furthermore, $\,u\,$ is a Kil\-ling 
field, so that $\,\langle u_t,x_s\rangle=\langle[\nabla\nh u]x_t,x_s\rangle
=-\hs\langle u_s,x_t\rangle$, while 
$\,\langle u,x_{st}\rangle=-\hs\langle u_t,x_s\rangle$, as (a) and 
(\ref{tps}.d) give $\,\langle u,x_s\rangle=0$. Consequently, 
$\,\langle u,x_t\rangle\nh_s\nh=\langle u,x_{st}\rangle
+\langle u,x_{ts}\rangle$, which yields (b), since $\,\nabla\,$ is 
tor\-sion-free, and so $\,x_{ts}\nh=x_{st}$. The relations just established 
and (a) also show that $\,\langle u,x_t\rangle\nh_s/2=\langle u,x_{st}\rangle
=-\hs\langle u_t,x_s\rangle=\langle u_s,x_t\rangle
=\langle[\nabla\nh u]x_s,x_t\rangle
=Q^{-1/2}\langle\nabla_{\!v}u,x_t\rangle$, which proves (c), as 
$\,\nabla_{\!v}u=\nabla_{\!v}(Jv)=J\nabla_{\!v}v=\ta\hskip.4ptJv
=\ta\hskip.4ptu\,$ by (\ref{not}) and (\ref{nvv}). Finally, (d) is obvious 
from (\ref{tps}.c) and (a).

By (c) and (d), $\,[\langle u,x_t\rangle/Q\hs]_s=0$. Hence 
$\,\langle u,x_t\rangle/Q\,$ is constant as a function of $\,s$. To show that 
$\,\langle u,x_t\rangle/Q=0$, we take the limit of 
$\,\langle u,x_t\rangle/Q\,$ as $\,s\to0$, that is (cf.\ 
Theorem~\ref{qfcot}(v)), as $\,\vp\to\vp_{\mathrm{min}}$. We may use 
l'Hospital's rule, since the numerator and denominator both vanish at 
$\,\vp=\vp_{\mathrm{min}}$ due to (\ref{not}) and the fact that (\ref{map}) 
has an obvious $\,C^\infty$ extension to $\,(b,c)\times[\hs0,\ds)$. Now (b), 
(d) and (a) give 
$\,\langle u,x_t\rangle\nh_s\hh/\nh Q_s
=\langle u,x_{st}\rangle\hs Q^{-1/2}/\ta
=\langle u/|u|,x_{st}\rangle/\ta$. The last expression tends to $\,0\,$ as 
$\,s\to0\,$ since $\,\ta=a\ne0\,$ at $\,\vp=\vp_{\mathrm{min}}$ due to 
Lemma~\ref{crlvl}(d) and (\ref{tps}.a), while $\,x_{st}$ at $\,s=0\,$ equals  
$\,\nabla_{\!\dot y}w$, and so $\,x_{st}\nh\to0\,$ as $\,s\to0$.

Consequently, $\,\langle u,x_t\rangle=0$, while $\,\langle v,x_t\rangle=0\,$ 
in view of (a) and the generalized Gauss lemma \cite[p.\ 26]{gray}. Therefore, 
$\,\mathcal{H}\,$ is the $\,F$-im\-age of $\,\hat{\mathcal{H}}$.

Combined with the assertion about $\,u\,$ in Lemma~\ref{fsutu} and 
(\ref{wwv}), this yields the formula for $\,\varOmega\,$ required by condition 
(vii) of Section~\ref{fe}, since, given sections $\,\hat w,\hat w\hh'$ of 
$\,\hat{\mathcal{H}}$, the $\,\hat{\mathcal{V}}\,$ component of 
$\,[\hat w,\hat w\hh'\hh]\,$ is 
$\,a^{-1}\varOmega(\hat w,\hat w\hh'\hh)\hs\hat u$, cf.\ 
\cite[formula (3.6)]{derdzinski-maschler-03}.

For fixed $\,t\in(b,c)$, let $\,\hat w\,$ be the 
$\,\hat{\mathcal{H}}$-hor\-i\-zon\-tal lift to 
$\,\pi^{-1}(\sd')\smallsetminus\sd'$ of a vector field on a neighborhood 
$\,\sd'$ of $\,y(t)\,$ in $\,\sd$, having the value $\,\dot y(t)\,$ at 
$\,y(t)$. As we just showed, the $\,F$-im\-age 
of $\,\hat w\,$ is a section $\,w\,$ of $\,\mathcal{H}$, defined on 
$\,F(\pi^{-1}(\sd')\smallsetminus\sd')$. Since $\,\hat w\,$ obviously commutes 
with $\,\hat v\,$ and $\,\hat u$, Lemmas~\ref{fstqv} and~\ref{fsutu} imply 
that $\,w\,$ commutes with $\,v\,$ and $\,u$, while, by (\ref{map}), 
$\,w_{x(t,s)}=x_t(t,s)\,$ for all $\,s\in(0,\ds)\,$ and our fixed $\,t$. 
Therefore (a) and (\ref{dvq}) give $\,[\hs\si\hs g(x_t,x_t)/Q\hh]_s\nh=0$. 
Thus, since $\,Q\hs/\nh(2\si)=\vp-\yj\,$ (see Lemma~\ref{dvgez}), 
$\,\langle x_t,x_t\rangle/(\vp-\yj)\,$ is constant as a function of $\,s$, 
that is, equal to its value at $\,s=0$. In other words, writing 
$\,y,\dot y,\yj,\vp\,$ instead of $\,y(t),\dot y(t),\yj(y(t))\,$ and 
$\,\vp(x(t,s))$, we have 
$\,\langle x_t,x_t\rangle
=(\vp_{\mathrm{min}}\nh-\yj)^{-1}(\vp-\yj)\langle\dot y,\dot y\rangle$, both 
if $\,\yj(y(t))\ne\infty$, and when $\,\yj(y(t))=\infty\,$ (provided that, 
in the latter case, one lets $\,(\vp_{\mathrm{min}}\nh-\yj)^{-1}(\vp-\yj)$ 
stand for $\,1$). In view of (\ref{het}), with $\,g\,$ now denoted by 
$\,\langle\,,\rangle$, the definition of $\,\hat g\,$ in Section~\ref{fe} 
thus shows that $\,\langle x_t,x_t\rangle\,$ at $\,(t,s)\,$ equals 
$\,\hat g(\hat w,\hat w)\,$ at $\,F^{-1}(x(t,s))$, proving our claim 
about $\,\hat g\,$ and $\,g$.

Finally, since $\,\dimr\sd=2$, both $\,\hat g\,$ and $\,g$, restricted to 
$\,\hat{\mathcal{H}}\,$ and $\,\mathcal{H}$, determine $\,\hat J\,$ on 
$\,\hat{\mathcal{H}}\,$ and $\,J\,$ on $\,\mathcal{H}\,$ uniquely up to a sign. Hence $\,F\,$ sends $\,\hat J\,$ on $\,\hat{\mathcal{H}}\,$ to $\,J\,$ 
on $\,\mathcal{H}$, with the plus sign due to the fact that 
$\,F=\mathrm{Id}\,$ on $\,\sd\,$ (which is tangent to both 
$\,\hat{\mathcal{H}}\,$ and $\,\mathcal{H}$).
\end{proof}
According to Lemmas~\ref{fstqv} --~\ref{fshgj}, $\,F\,$ is a 
bi\-hol\-o\-mor\-phic isometry of $\,(\hat M'\nnh,\hat g)\,$ onto 
$\,(M'\nnh,g)$, sending the Kil\-ling potential $\,\vp\,$ on 
$\,(\hat M'\nnh,\hat g)\,$ to $\,\vp\,$ on $\,(M'\nnh,g)$. Lemma~\ref{ismex} 
now implies that $\,F\,$ has an extension $\,M'\nh\to M$, which proves 
Theorem~\ref{clssf}.

\end{document}